%% file: main.tex
\documentclass[onefignum,onetabnum]{siamart250211}

\input{ex_shared}

\ifpdf
\hypersetup{
  pdftitle={Delta-VFE pivoted Cholesky},
  pdfauthor={L. Schaub and P. Zaspel}
}
\fi

\myexternaldocument{ex_supplement}

\begin{document}
\maketitle

\input{Sections/00_abstract}
\input{Sections/01_intro}
\input{Sections/02_basics}
\input{Sections/02_PCholesky}
\input{Sections/04_Objective_Aware_Pivoting}
\input{Sections/05_Efficient_Realisation}

\input{Sections/06_Interpretation}
\input{Sections/07_experiments}

\input{Sections/08_conclusion_outlook}
\newpage
\input{Sections/acknowledgment}
\bibliographystyle{siam}
\bibliography{review}
\appendix
\input{Sections/09_appendix}

\newpage
\newpage
\newpage

\end{document}

%% file: ex_shared.tex
\usepackage{bm}
\usepackage{hyperref}
\usepackage{caption}
\usepackage{subcaption}
\usepackage{aligned-overset}

\usepackage{lipsum}
\usepackage{amsfonts}
\usepackage{graphicx}
\usepackage{epstopdf}
\usepackage{algorithmic}

\usepackage[textwidth=2.8cm]{todonotes}
\usepackage{etoolbox}
\makeatletter
\patchcmd{\@addmarginpar}{\ifodd\c@page}{\iftrue}{}{}
\makeatother

\newcommand{\R}{\mathbb{R}}

\newcommand{\tr}{\mathrm{tr}}

\DeclareMathOperator*{\argmax}{argmax}

\ifpdf
  \DeclareGraphicsExtensions{.eps,.pdf,.png,.jpg}
\else
  \DeclareGraphicsExtensions{.eps}
\fi

\newsiamremark{remark}{Remark}
\newsiamremark{hypothesis}{Hypothesis}
\crefname{hypothesis}{Hypothesis}{Hypotheses}
\newsiamthm{claim}{Claim}

\headers{$\Delta$-VFE Pivoted Cholesky}{L. Schaub and P. Zaspel}

\title{Variational Free Energy Pivot Selection for Pivoted Cholesky}

\author{Louise Schaub\thanks{School of Mathematics and Natural Sciences, University of Wuppertal (\email{schaub@uni-wuppertal.de}, \email{zaspel@uni-wuppertal.de}).}
\and Peter Zaspel\footnotemark[1]}

\usepackage{amsopn}
\DeclareMathOperator{\diag}{diag}

\makeatletter
\newcommand*{\addFileDependency}[1]{
  \typeout{(#1)}
  \@addtofilelist{#1}
  \IfFileExists{#1}{}{\typeout{No file #1.}}
}
\makeatother

\newcommand*{\myexternaldocument}[1]{%
    \externaldocument{#1}%
    \addFileDependency{#1.tex}%
    \addFileDependency{#1.aux}%
}

%% file: Sections/00_abstract.tex
\begin{abstract}
Pivoted Cholesky factorizations construct low-rank approximations of symmetric positive definite matrices by sequentially selecting pivots from the residual diagonal. Classical greedy and randomized rules, such as randomly pivoted Cholesky, target the algebraic trace-norm error of the residual. In many applications, however, the matrix enters a nonlinear matrix functional whose value, not the trace-norm error, determines solution quality, and residual-based rules ignore this structure. 
We derive a pivot rule that maximizes the exact one-step change of such a functional under Cholesky-consistent rank-1 updates, for a functional combining log-determinant, quadratic, and trace terms. 
This functional arises as the variational free energy in Gaussian process regression, where the matrix is a kernel matrix. The resulting per-step gain admits a closed-form additive decomposition into complexity, data-fit, and trace contributions, and is used directly as a pivot-selection criterion. We refer to the resulting method as $\Delta$-VFE pivoted Cholesky. At each iteration, the criterion is evaluated on a batch of $s$ candidate pivots sampled proportionally to the residual diagonal via incremental Woodbury updates, at a total cost of $\mathcal{O}(snr^2)$ for an $n\times n$ matrix and target rank $r$. This matches the asymptotic complexity of randomly pivoted Cholesky up to the batch factor $s$. Cholesky-consistent rank-1 updates yield monotonically non-decreasing functional values, and the proposed rule maximizes the per-step gain among them. Numerical experiments show improved objective values and predictive accuracy at low to moderate ranks compared to classical and randomly pivoted Cholesky, while preserving trace-norm approximation quality.
\end{abstract}

%% file: Sections/01_intro.tex
\section{Introduction}
Low-rank approximations of symmetric positive definite matrices are a central tool in numerical linear algebra and scientific computing. Among these, pivoted Cholesky factorizations \cite{harbrecht.2012, higham.1990, higham.2009} provide an efficient and incremental mechanism for constructing approximations of the form $K \approx \widetilde{K}= LL^\top$ by selecting pivots sequentially from the diagonal of the residual matrix. Classical greedy pivoted Cholesky and its randomized variant RPCholesky~\cite{chen.2025} select pivots according to residual diagonal magnitudes, a strategy that is well understood and closely connected to Nystr\"om-type approximations~\cite{halko.2011}. 

In many applications, however, the matrix $K$ does not appear in isolation, but rather through nonlinear matrix functionals of its low-rank approximation $\widetilde{K}$. In such settings, residual-based pivot rules are motivated solely by the approximation error $\tr(K - \widetilde{K})$ \cite{gittens.2013}. When the relevant functional also includes other terms, this criterion neglects how the approximation interacts with the associated operator structure, which can lead to systematically suboptimal approximations.

This mismatch is particularly pronounced in Gaussian process regression, where the natural quality measure is the marginal log-likelihood $\mathcal{F}(K)$, which costs $\mathcal{O}(n^3)$ for the number of data points $n$ to evaluate exactly. When $K$ is replaced by a rank-$r$ low-rank approximation $\widetilde{K} \preceq K$, the tightest computable surrogate is the \textit{variational free energy} (VFE) $\mathcal{L}(\widetilde{K})$ as introduced by Titsias~\cite{Titsias.2009}. This poses a lower bound on $\mathcal{F}(K)$, that is tight when $\widetilde{K} = K$ and combines log-determinant, quadratic data-fit, and trace components. Of these, only the trace term is targeted by residual-based pivot rules. 
Therefore, the quadratic term, which depends on the data vector $\bm{y}$, and the log-determinant term are not targeted by any residual-based pivot selection rule. 
This raises the following question: \\ \textit{How should pivot selection be performed when the goal is to optimize a matrix functional rather than purely algebraic approximation error?}

In this work, we answer this question by deriving the exact change in $\mathcal{L}(\widetilde{K})$ induced by a rank-1 update that preserves positive definiteness and Cholesky consistency. This leads to a new greedy pivot selection method, which we call \textit{$\Delta$-VFE pivoted Cholesky}. At each iteration, it selects the candidate maximizing the resulting gain in $\mathcal{L}(\widetilde{K})$. The resulting per-update gain admits a closed-form expression combining a log-determinant contribution, a quadratic data-alignment term, and a trace-reduction term. Despite the non-algebraic structure of the gain formula, we show how it can be evaluated for $s$ candidates and $n$ data points in $\mathcal{O}(snr^2)$ total cost via incremental Woodbury updates, matching the asymptotic complexity of RPCholesky, $\mathcal{O}(nr^2)$, up to the batch factor $s$. We further prove that the resulting pivot rule produces a monotonically non-decreasing sequence of $\mathcal{L}(\widetilde{K})$ values.

We validate the approach on benchmark problems, analyzing three measures. First, the proposed pivot rule yields improved values of the VFE bound $\mathcal{L}(\widetilde{K})$, with smaller deviations from $\mathcal{F}(K)$ at lower ranks. Second, the trace-norm approximation error remains comparable to that of RPCholesky throughout, so the gains in the first measure come at no cost in standard kernel approximation quality. Third, this gain in functional value translates into reduced solution error of the regularized linear system arising in Gaussian process regression at low to moderate ranks. 

\subsection{Related Work}\label{subsec:previousWork}

We briefly review prior work on low-rank kernel approximation, pivot selection strategies, and sparse Gaussian process inference, which together form the foundation of our approach.

\paragraph{Pivoted Cholesky and Low-Rank Approximation}
The pivoted Cholesky decomposition is a classical method for constructing low-rank approximations of symmetric positive-definite matrices. Harbrecht et al.~\cite{harbrecht.2012} analyzed its computational complexity, establishing $\mathcal{O}(nr^2)$ cost and convergence in trace norm. Their analysis highlights that the standard diagonal pivot rule provides a computationally efficient surrogate for the exact trace-reduction maximizer, a distinction we make explicit in Lemma~\ref{lem:trace}. 

Randomized variants of pivoted Cholesky have recently been proposed, including \emph{RPCholesky} by Chen et al.~\cite{chen.2025}, which samples pivots proportionally to the residual diagonal and achieves near-optimal trace-norm guarantees. Epperly et al.~\cite{epperly.2025} further developed accelerated blocked variants with significant practical speedups. A broader overview of randomized low-rank approximation methods is given by Martinsson and Tropp~\cite{martinsson.2020}. 

In the context of kernel methods, Williams and Seeger~\cite{williams.2000} introduced \emph{Nystr\"om approximations}, which construct low-rank approximations via subsampling. Gittens and Mahoney~\cite{gittens.2013} analyzed the approximation quality of Nystr\"om methods under various sampling strategies and established conditions for near-optimal performance. 

One of these strategies being \emph{ridge leverage scores} sampling, as introduced and examined in~\cite{alaoui.2015, musco.2017}. Ridge leverage score sampling selects columns proportionally to the diagonal of $K(K + \sigma_\varepsilon^2 I)^{-1}$, capturing global point importance under the full kernel matrix. 

In contrast, our proposed method is not only using the matrix and its properties to build a low-rank approximation, but additionally includes a complexity term and a data-fit term.

\paragraph{Objective-Aware and Data-Aware Pivot Selection}
Several works have explored incorporating task-dependent information into pivot selection. Bach and Jordan~\cite{bach.2005} proposed the \emph{Cholesky with side information} (CSI) algorithm, which selects pivots by approximately minimizing a combined objective balancing kernel approximation error and predictive performance. However, evaluating the exact gain in this objective is computationally prohibitive, and the method relies on approximations and look-ahead strategies. Bach~\cite{bach.2013} provides a sharp analysis of low-rank kernel approximations in supervised learning, showing that the rank required to preserve predictive performance scales with the degrees of freedom of the problem. A key insight of their work is that accurate prediction does not require accurate approximation of the kernel matrix itself, but rather alignment with the downstream learning objective. This perspective complements classical trace-based analyses and motivates the design of approximation methods that are directly guided by predictive criteria.

More recently, Schreiter et al.~\cite{schreiter.2016} proposed selecting inducing points, i.e.~the points used to construct the low-rank approximation, using a maximum-error criterion that prioritizes locations with large predictive residuals under the current model. Specifically, at each step their rule selects a candidate maximizing the pointwise predictive error, providing a scalar, magnitude-based score per candidate. 

While these approaches incorporate additional structure into pivot selection, they are typically motivated by surrogate criteria or specific components of the objective rather than derived directly from a global inference objective. Our proposed method, on the other hand, derives an exact per-step gain directly from the variational free energy of the Gaussian process model, yielding a greedy pivot rule that is explicitly aligned with the underlying probabilistic inference objective.

\paragraph{Sparse Gaussian Process Approximation}
\sloppy
Sparse Gaussian process inference aims to reduce the computational cost of Gaussian process models by constructing low-rank approximations based on a subset of inducing variables or inducing points~\cite{liu.2020}. Variational formulations for sparse Gaussian process inference were introduced by Titsias~\cite{Titsias.2009}, who showed that the collapsed VFE provides a lower bound on the log marginal likelihood and admits closed-form optimization over the variational distribution. Qui\~{n}onero-Candela et al.~\cite{quionero.2005} provided a unifying framework for sparse GP methods. 

Burt et al.~\cite{burt.2020} derived convergence rates for VFE approximations and showed that appropriately chosen inducing point sets can achieve near-optimal rates, with greedy constructions such as pivoted Cholesky providing effective practical approximations. Hensman et al.~\cite{hensman.2013} developed scalable optimization of VFE-based models via stochastic methods.

\paragraph{Variational Free Energy \& Pivot Selection}
Most closely related to our work is the \textit{CholQR} algorithm of Cao et al.~\cite{cao.2015}, developed for sparse Gaussian process regression. They derive an exact closed-form expression for the per-candidate change in the VFE under a one-step partial Cholesky augmentation represented through an augmented QR factorization, with the gain decomposed into data-fit, complexity, and trace terms. A scalar quantity equivalent to the per-step gain in our Theorem~\ref{thm:per-pivot-gain} appears in their work. Rather, our contribution is to reinterpret and re-derive this increment within the framework of pivoted Cholesky as an intrinsic rank-1 factorization functional associated with admissible Cholesky-consistent updates.

The algorithmic settings differ substantially. CholQR is designed for sparse GP inducing-point optimization and interleaves local-search swaps with gradient-based hyperparameter updates. Candidate points are screened by a rank-$r$ partial Cholesky surrogate of the residual, with the exact variational gain evaluated only for the proposed swap. An incremental forward-greedy variant is described in their appendix and reported not to improve performance in that setting. In contrast, we work in the fixed-hyperparameter pivoted Cholesky regime, where the primary object is a rank-$r$ kernel factorization for downstream numerical approximation and inference tasks rather than a compact inducing set jointly optimized with model parameters. Candidates are sampled proportionally to the residual diagonal, recovering RPCholesky~\cite{chen.2025} exactly at batch size $s=1$ and thereby connecting the method directly to the randomized pivoted-Cholesky literature. This reframing within a hierarchy of pivot rules together with greedy pivoted Cholesky and RPCholesky, with each level incorporating successively more information about the objective.

\subsection{Outline}
The remainder of the paper is organized as follows. Section~\ref{sec:basics} introduces notation and reviews the necessary background on kernel matrices, Gaussian process regression, and the main functional of this work, the variational free energy. Section~\ref{sec:pcholesky} reviews pivoted Cholesky and establishes the classical baseline, as well as its randomized variant. We derive our proposed pivot rule, $\Delta$-VFE pivoted Cholesky, and the per-pivot gain formula in Section~\ref{sec:objectiveAware} and  develop the algorithmic realization, including Woodbury-based inverse updates and complexity analysis in Section~\ref{sec:realization}. This is followed by an interpretation of the structure of the gain formula and relate it to existing methods in Section~\ref{sec:interpretation}. We present numerical experiments, analyzing the objective performance as well as the approximation and prediction error in Section~\ref{sec:experiments}.

%% file: Sections/02_basics.tex
\section{Review of Basics and Terminology}\label{sec:basics}
In the following, we briefly revisit the necessary basics and define the quantities that are used in this work. 

\paragraph{Kernels}
Since Gaussian processes rely on positive definite kernels, we start with formally defining this notion.
\begin{definition}[Symmetric Positive Definite Kernel \cite{schaback.2006}]\label{def:pdKernel}
     A function $k:\R^d\times \R^d\to\R$ is called \emph{symmetric positive definite kernel} (spd) if 
     \begin{equation*}
         k(\bm{x}_i, \bm{x}_j) = k(\bm{x}_j, \bm{x}_i) 
         \quad 
        \text{and}
        \quad
        \sum_{i=1}^n \sum_{j=1}^n c_i c_j k(\bm{x}_i,\bm{x}_j) > 0
    \end{equation*}
    for any set of $n$ (unique) points $\bm{x}_i,\,\bm{x}_j \in \R^d$ and any choice of numbers $c_i \in \R$. 
\end{definition}
In the following we will refer to spd kernels simply as \textit{kernels}.
From this definition we can further define the kernel matrix. 
\begin{definition}[Kernel Matrix \cite{scholkopf.2002}]\label{def:kernelMatrix}
    Let $k : \mathbb{R}^d \times \mathbb{R}^d \to \mathbb{R}$ be a kernel and let $X := \{\bm{x}_1, \ldots, \bm{x}_n\} \subset\mathbb{R}^d$ with $n \in \mathbb{N}$ be given. The matrix $K \in \mathbb{R}^{n \times n}$ with
    \begin{equation}
        K_{ij} := k(\bm{x}_i, \bm{x}_j), \quad i, j = 1, \ldots, n
    \end{equation}
    is called \emph{kernel matrix} of $k$ with respect to $X$.
\end{definition}
We proceed by defining the kernels of interest for this work. For simplicity we restrict our analysis to the Gaussian and Laplacian kernel. Further kernels can be found e.g.~in \cite{CarlEdwardRasmussenandChristopherK.I.Williams., duvenaud.2014}.   
\begin{definition}[Kernels \cite{CarlEdwardRasmussenandChristopherK.I.Williams., scholkopf.2002}]\label{def:kernels}
    Let $\bm{x},\bm{x}'\in\R^d$ and $\ell>0$. The \emph{Gaussian kernel} is defined as 
    \begin{equation}
        k(\bm{x}, \bm{x}') := \exp \left( -\frac{\|\bm{x}-\bm{x}'\|_2^2}{2\ell^2}\right). \label{eq:gaussianKernel}
    \end{equation}
    The \emph{Laplacian kernel} is defined as 
    \begin{equation}
        k(\bm{x}, \bm{x}') := \exp \left( -\frac{\|\bm{x}-\bm{x}'\|_1}{\ell}\right), \label{eq:laplacianKernel}
    \end{equation}
\end{definition}

\paragraph{Gaussian Process Regression}
We give the most important details on the source of our matrix functional. More details on \textit{Gaussian process regression} (GPR) can be found e.g.~in \cite{CarlEdwardRasmussenandChristopherK.I.Williams.}.

Let $X = \{\bm{x}_1,\dots,\bm{x}_n\} \subset \mathbb{R}^d$ and let $K \in \mathbb{R}^{n\times n}$ denote the associated kernel matrix. We consider scalar observations $y_i \in \mathbb{R}$ modeled as
\begin{equation*}
    y_i = f(\bm{x}_i) + \varepsilon_i,
    \qquad
    \varepsilon_i \overset{\mathrm{i.i.d.}}{\sim} \mathcal{N}(0,\sigma_\varepsilon^2),
\end{equation*}
where $f \sim \mathcal{GP}(0,k)$ and the noise variables $\varepsilon_i$ are independent of $f$. Let $\bm{y} = (y_1,\dots,y_n)^\top$. Then the vector of unknown function values $\bm{f} := (f(\bm{x}_1),\dots,f(\bm{x}_n))^\top$ follows the Gaussian distribution $\mathcal{N}(0,K)$. Since $\bm{y} = \bm{f} + \boldsymbol{\varepsilon}$, and $\boldsymbol{\varepsilon} \sim \mathcal{N}(0,\sigma_\varepsilon^2 I),$ it follows that
\begin{equation*}
    \bm{y} \sim \mathcal{N}(0, K + \sigma_\varepsilon^2 I).
\end{equation*}
We briefly recall the predictive quantities associated with GPR, since they will be used later in the interpretation of the data-alignment term and in the numerical experiments, i.e.~Sections~\ref{sec:interpretation} and~\ref{sec:experiments}.
\begin{definition}[Predictive posterior]
    Let \(X_\star=\{\bm{x}_1^\star,\hdots,\bm{x}_m^\star\}\subset\mathbb{R}^d\) denote test data points and define
    $(K_{\star X})_{ij}:=k(\bm{x}_i^\star,\bm{x}_j)$ and $(K_{\star\star})_{ij}:=k(\bm{x}_i^\star,\bm{x}_j^\star).$
    Under the Gaussian process regression model, the conditional distribution of the unknown function values at the test data points $X_\star$, given the observations $\bm{y}$, is Gaussian with \emph{mean}
    \begin{equation*}
    \bm\mu_\star = K_{\star X}(K+\sigma_\varepsilon^2 I)^{-1}\bm y
    \end{equation*}
    and \emph{covariance}
    \begin{equation*}
    \Sigma_\star
    =
    K_{\star\star}
    -
    K_{\star X}(K+\sigma_\varepsilon^2 I)^{-1}K_{X\star}.
    \end{equation*}
\end{definition}
In particular, on the training data point set itself, $\bm\mu = K(K+\sigma_\varepsilon^2 I)^{-1}\bm y.$
For a low-rank approximation \(\widetilde K_i\), we define the approximate training-set posterior mean
$\widetilde{\bm\mu}_i:=\widetilde K_i(\widetilde K_i+\sigma_\varepsilon^2 I)^{-1}\bm{y}.$

A standard objective for model selection in GPR is the log marginal likelihood, which forms the starting point for the variational free energy introduced below.
\begin{definition}[Log Marginal Likelihood]
    Under the above Gaussian model, the \emph{log marginal likelihood} (LML) is given by
    $\mathcal{F}(K):=\log \mathcal{N}(\bm y \mid 0,\, K + \sigma_\varepsilon^2 I),$
    which can be written explicitly as
    \begin{equation}
        \label{eq:lml}
        \mathcal{F}(K)
        =
        -\frac12
        \left(
        \log |K+\sigma_\varepsilon^2 I|
        +
        \bm y^\top (K+\sigma_\varepsilon^2 I)^{-1}\bm y
        +
        n\log(2\pi)
        \right),
    \end{equation}
where $|\cdot|$ denotes the determinant.
\end{definition}
The LML measures how well the probabilistic model explains the observed data while automatically balancing data fit and model complexity. In GPR, it is commonly used as an objective for determining the kernel parameters, e.g.~$\ell$ and $\sigma_\varepsilon^2$, and model selection, since maximizing the LML corresponds to selecting the covariance structure that renders the observed data most probable under the model. However, evaluating~\eqref{eq:lml} costs $\mathcal{O}(n^3)$ due to the matrix inversion of a dense covariance matrix. This motivates the use of low-rank approximations in large-scale settings. 
A widely used approach is based on Nystr\"om approximations and variational inference. This leads to the \emph{variational free energy}~\cite{Titsias.2009}, which poses a lower bound on $\mathcal{F}(K)$ achievable with a rank-$r$ low-rank approximation  $\widetilde{K} = LL^\top \preceq K$.
\begin{definition}[Variational Free Energy]
    Let $K \in \R^{n\times n}$, $\bm{y} \in \mathbb{R}^n$, and $\sigma_\varepsilon^2 > 0$. For a low-rank approximation $\widetilde{K} \in \R^{n\times n}$ of $K$ with $\widetilde{K} \preceq K$, the \emph{variational free energy} is defined as
    \begin{equation}
        \label{eq:vfe}
        \mathcal{L}(\widetilde{K})
        :=
        -\frac{1}{2}\left(
        \log|\widetilde{K}+\sigma_\varepsilon^2 I|
        + \bm{y}^\top(\widetilde{K}+\sigma_\varepsilon^2 I)^{-1}\bm{y}
        + \frac{1}{\sigma_\varepsilon^2}\tr(K-\widetilde{K})
        + n\log(2\pi)
        \right).
    \end{equation}
\end{definition}
The three included terms, beside the constant, are described as the complexity term, the data-fit term and a trace penalty term \cite{vanderwilk.2019}. 

In the remainder of this work we will use $\mathcal{L}(\widetilde{K})$ and analyse its per step change, from which we will then derive a pivot selection criterion for an objective-aware pivoted Cholesky decomposition.

%% file: Sections/02_PCholesky.tex
\section{Pivoted Cholesky} \label{sec:pcholesky}

We review pivoted Cholesky factorizations and establish the classical baseline against which our objective-aware pivot rule targeting $\mathcal{L}(\widetilde{K})$ of Section~\ref{sec:objectiveAware} is  compared. We present the standard greedy and randomized pivot rules, characterize the exact one-step trace reduction induced by an arbitrary pivot, and clarify the sense in which the greedy rule acts as a tractable proxy for this exact reduction. The section closes by specializing the framework to kernel matrices, which provides the setting for the remainder of our work.

In the following, let $A\in\R^{n\times n}$ be an spd matrix. Pivoted Cholesky \cite{higham.1990, higham.2009, harbrecht.2012} builds the low-rank factor $L_r$ of $A$ incrementally. At each iteration $i$, a pivot index $p$ is selected and a new column is appended to $L_i$ via one step of the Cholesky factorization applied to the current residual. Let
\begin{equation}
\label{eq:residual}
    T_i := A - L_i L_i^\top \succeq 0
\end{equation}
denote the \textit{residual} after $i$ steps, and let $\bm{d}^{(i)} := \mathrm{diag}(T_i)$ be its diagonal.
We write $\bm{t}_p := (T_i)_{:,p}$ for the $p$-th column of the residual and $d_p^{(i)} := (T_i)_{p,p}$ for the $p$-th diagonal entry, so that $\bm{d}^{(i)} = (d_1^{(i)}, \ldots, d_n^{(i)})^\top$. 
The rank-1 update induced by selecting pivot $p$ takes the Schur complement form
\begin{equation}
    T_{i+1} = T_i - \frac{(T_i)_{:,p}\,(T_i)_{p,:}}{(T_i)_{p,p}}
    = T_i - \bm{u}_p \bm{u}_p^\top,
    \quad\text{with}\quad
    \bm{u}_p := \frac{\bm{t}_p}{\sqrt{d_p^{(i)}}},
\label{eq:schur}
\end{equation}

This strategy zeros out the $p$-th diagonal entry and preserves positive definiteness, i.e.\ $T_{i+1} \succeq 0$. After $r$ steps this yields $A \approx L_r L_r^\top$ with $L_r \in \mathbb{R}^{n \times r}$ at $\mathcal{O}(n r^2)$ total cost~\cite{harbrecht.2012}.

\begin{remark}
    Different conventions appear in the literature. Harbrecht et al.~\cite{harbrecht.2012} work with a square lower-triangular factor $\widetilde{L}_r$ obtained via explicit symmetric row/column permutations, giving the factorization $P^\top A P = \widetilde{L}_r \widetilde{L}_r^\top$. In the randomized linear algebra and Nystr\"om literature~\cite{chen.2025, gittens.2013, martinsson.2020}, the approximation is instead represented either via a tall-thin factor $L_r \in \mathbb{R}^{n \times r}$ with rows in the original indexing, or directly through the column-subset form $K(:,Z) K(Z,Z)^\dagger K(Z,:)$ for a pivot set $Z$. In neither case is a permutation matrix carried through the algorithm. The two viewpoints are related by $L_r = P\widetilde{L}_r$, where $P$ brings the pivot set $Z$ to the leading block. We adopt the tall-thin form throughout to avoid carrying $P$ through subsequent derivations. The only place where the permutation view is invoked is Remark~\ref{rem:nystrom}, where identifying $T_i$ as a Schur complement of $K$ with respect to $Z$ requires viewing $Z$ as a leading block.
    \hfill$\triangle$
\end{remark}

\paragraph{The Greedy Pivot Rule}
Standard realizations of the pivoted Cholesky factorization select the next pivot as the index of the largest remaining diagonal residual~\cite{harbrecht.2012},
\begin{equation}
\label{eq:greedy}
    p_\star \in \argmax_p\; (T_i)_{p,p},
\end{equation}
motivated by the observation that $(T_i)_{p,p}$ measures the residual diagonal mass at index $p$ under the current approximation.
The resulting algorithm is stated as~Algorithm~\ref{alg:pivotedCholesky}, as the ``\textsc{Greedy}'' mode.

The following lemma gives the exact one-step decrease in $\tr (T_i)$ induced by any candidate pivot, and identifies the pivot that maximizes it. The exact one-step decrease will play a key role in Section~\ref{sec:objectiveAware}, where it appears as one of the three components of our per-update gain formula.

\begin{lemma}[One-Step Trace Decrease]
\label{lem:trace}
Let $T_i$ be symmetric and let $p$ satisfy $(T_i)_{p,p} > 0$.
Then,
\begin{equation*}
\label{eq:trace_decrease}
    \mathrm{tr}(T_{i+1})
    =
    \mathrm{tr}(T_i)
    -
    \frac{(T_i^2)_{p,p}}{(T_i)_{p,p}}.
\end{equation*}
Consequently,
\begin{equation}
\label{eq:trace_maximiser}
    p_\star \in \argmax_p\; \frac{(T_i^2)_{p,p}}{(T_i)_{p,p}}
\end{equation}
maximizes the one-step decrease of $\mathrm{tr}(T_i)$.
\end{lemma}

\begin{proof}
From~\eqref{eq:schur}, $T_{i+1} = T_i - \bm{u}_p \bm{u}_p^\top$. By linearity of the trace and $\mathrm{tr}(\bm{u}_p \bm{u}_p^\top) = \|\bm{u}_p\|_2^2$,
\begin{equation*}
\mathrm{tr}(T_{i+1})
= \mathrm{tr}(T_i) - \|\bm{u}_p\|_2^2
= \mathrm{tr}(T_i) - \frac{\|\bm{t}_p\|_2^2}{(T_i)_{p,p}}.
\end{equation*}
Since $T_i$ is symmetric,
\begin{equation*}
    \|T_i \bm{e}_p\|_2^2
    = (T_i)_{:,p}^\top(T_i)_{:,p}
    = (T_i^2)_{p,p},
\end{equation*}
which yields the result.
\end{proof}

\begin{remark}
\label{rem:greedy_proxy}
The greedy rule~\eqref{eq:greedy} selects the largest diagonal entry $(T_i)_{p,p}$. In contrast, the exact trace-reduction maximizer in~\eqref{eq:trace_maximiser} depends on $(T_i^2)_{p,p}/(T_i)_{p,p} = \|\bm{u}_p\|_2^2$, which requires access to the full $p$-th column of $T_i$. Harbrecht et al.~\cite{harbrecht.2012} observe that finding the pivot that optimally reduces the trace requires knowledge of the complete matrix, and justify the diagonal rule as a tractable substitute. The justification follows from the positive semi-definiteness of $T_i$, which yields the entrywise bound $|(T_i)_{j,p}|^2 \leq (T_i)_{j,j}\,(T_i)_{p,p}$ and consequently
\begin{equation*}
    \|\bm{u}_p\|_2^2 = \frac{(T_i^2)_{p,p}}{(T_i)_{p,p}} = \frac{1}{(T_i)_{p,p}} \sum_{j=1}^n (T_i)_{j,p}^2 \leq \mathrm{tr}(T_i),
\end{equation*}
so that the diagonal entry $(T_i)_{p,p}$ serves as a computationally cheap proxy for $\|\bm{u}_p\|_2^2$. \hfill$\triangle$
\end{remark}

\paragraph{Randomized Variant}
\textit{Randomly pivoted Cholesky} (RPCholesky), introduced and refined in~\cite{chen.2025,epperly.2025}, replaces the deterministic selection~\eqref{eq:greedy} by sampling pivot indices without replacement proportionally to the residual diagonal $d^{(r)}_p$ at iteration $r$. This yields a stochastic analogue of the diagonal-based selection rule, with provable trace-norm guarantees that hold for arbitrary spd matrices~\cite{chen.2025}. The randomized variant achieves the same $\mathcal{O}(nr^2)$ asymptotic complexity as the deterministic rule.
Algorithm~\ref{alg:pivotedCholesky} summarizes this variant as the ``\textsc{Randomized}'' mode.

\begin{remark}[Specialization to Kernel Matrices and the Nystr\"om Connection]
\label{rem:nystrom}
Throughout the remainder of this work, we specialize the spd matrix $A$ to a kernel matrix $K \in \mathbb{R}^{n \times n}$ arising from a positive-definite kernel evaluated on a point set $X = \{x_1, \ldots, x_n\}$ as defined in Section~\ref{sec:basics}. In this setting, the low-rank factor produced by Algorithm~\ref{alg:pivotedCholesky} corresponds to a Nystr\"om-type approximation~\cite{williams.2000} of rank $r$
\begin{equation*}
    \widetilde{K}_r = L_r L_r^\top = K_{XZ} K_{ZZ}^{-1} K_{ZX},
\end{equation*}
where $Z \subset X$ is the set of selected pivots. This connection is stated e.g.~in~\cite[Section 5.4]{halko.2011}. After reordering indices so that $Z$ occupies the leading block, the residual $T_i = K - \widetilde{K}_r$ coincides with the Schur complement of $K$ with respect to $Z$, and  $\mathrm{tr}(T_i)$ measures the total unexplained variance.
Section~\ref{sec:objectiveAware} builds on this  specialization to derive an objective-aware pivot rule targeting a matrix functional that arises naturally from $\mathcal{L}(\widetilde{K})$ in this setting.\hfill$\triangle$
\end{remark}

\begin{algorithm}[t]
\caption{\textsc{(Randomly) Pivoted Cholesky}} \label{alg:pivotedCholesky}
\begin{algorithmic}[1]
\REQUIRE spd matrix $A \in \mathbb{R}^{n \times n}$,
target rank $r$, mode $\in \{\textsc{greedy}, \textsc{random}\}$
\ENSURE Low-rank factor $L \in \mathbb{R}^{n \times r}$ such that $A \approx LL^\top$
\STATE $\bm{d} \leftarrow \operatorname{diag}(A)$ \COMMENT{Residual diagonal}
\STATE $L \leftarrow [\,]$
\FOR{$j = 1, \dots, r$}
    \IF{mode = \textsc{greedy}}
        \STATE $p \leftarrow \argmax_i\, d_i$
        \COMMENT{Greedy pivot}
    \ELSIF{mode = \textsc{random}}
        \STATE $\displaystyle w_i \leftarrow \frac{{d}_i}{\sum_{\ell} {d}_\ell}$
        \STATE Sample $p$ from $\{1,\dots,n\}$ using probabilities $\{w_i\}$     
        \COMMENT{Random pivot proportional to residual}
    \ENDIF
    \STATE $\bm{v} \leftarrow A_{:,p} - L_{:,\,1:j-1}\,(L_{p,\,1:j-1})^\top$ 
    \STATE $L_{:,j} \leftarrow \bm{v} / \sqrt{v_p}$ 
    \STATE $\bm{d} \leftarrow \bm{d} - L_{:,j}^2$ 
    \COMMENT{Update residual diagonal}
\ENDFOR
\RETURN $L$
\end{algorithmic}
\end{algorithm}

%% file: Sections/04_Objective_Aware_Pivoting.tex
\section{Objective-Aware Pivoting}
\label{sec:objectiveAware}

In this section we derive the $\Delta$-VFE pivoted Cholesky pivot rule resulting from the exact per-step change of the VFE, which we call \textit{$\Delta$-VFE pivoted Cholesky}. We do so by maximizing the exact one-step improvement of $\mathcal{L}(\widetilde{K}) = -\frac{1}{2}\left( \log|\widetilde{K}+\sigma_\varepsilon^2 I| + \bm{y}^\top (\widetilde{K}+\sigma_\varepsilon^2I )^{-1}\bm{y} + \frac{1}{\sigma_\varepsilon^2}\tr (K - \widetilde{K}) + n\log(2\pi) \right)$ over Cholesky-consistent rank-1 updates. In the following, we drop the constant term $n\log(2\pi)$. Although this functional originates from variational inference in GPR, our focus here is on its numerical structure. Our proposed per-pivot gain formula balances contributions from the three terms described in Section~\ref{sec:basics}, i.e.~covariance complexity, data alignment, and residual variance. 

Since $\mathcal{L}(\widetilde{K}) \leq \mathcal{F}(K)$, we seek a pivot rule that targets the maximization of it in each iteration.
For the derivation of our per-pivot gain formula, we first analyze the elements of iterations $i$ and $i+1$. 
Therefore, let $\widetilde K_i=L_iL_i^\top$ be a rank-$i$ approximation of $K$ and define the residual, diagonal element and regularized covariance matrix respectively as
\begin{equation*}
    T_i := K-\widetilde K_i \succeq 0, 
    \quad
    \bm{d}^{(i)} := \diag(T_i)
    \quad 
    \text{and}
    \quad 
    \Sigma_i := \widetilde{K}_i + \sigma_\varepsilon^2 I.
\end{equation*}
At iteration $i$, we evaluate candidate rank-1 updates of the form
\begin{equation}
    \widetilde K_{i+1} := \widetilde K_i + \bm{u} \bm{u}^\top, 
    \quad 
    \text{where}
    \quad
    \bm{u} \in \mathbb R^n,
\label{eq:candidate-update}
\end{equation}
which induces the updates
\begin{align*}
    \Sigma_{i+1} = \Sigma_i + \bm{u} \bm{u}^\top, \quad \text{and}\quad
    T_{i+1} = T_i - \bm{u} \bm{u}^\top.
\end{align*}
Before we continue with the exact per-step change of $\mathcal{L}(\widetilde{K})$, we introduce three auxiliary lemmas. The first one allows a cheap approach of computing a matrix determinant after a rank-1 update, given we have computed the determinant beforehand.
\begin{lemma}[Matrix Determinant Lemma {\cite[Chapter 18, Section 1]{harville.1998}}] \label{lem:matrixDeterminant}
    Let $A\in\R^{n\times n}$ be a regular matrix and $\bm{u}, \bm{v} \in \R^{n}$ be two column vectors. Then
    \begin{equation*}
        |A + \bm{u}\bm{v}^\top| = |A| (1+\bm{v}^\top A^{-1} \bm{u}).
    \end{equation*}
\end{lemma}
The second lemma focuses on inverses of rank-1 updates to a matrix, where the inverse of the matrix has been computed already.
\begin{lemma}[Sherman-Morrison Formula {\cite[Section 2.1.4]{golub.2013}}]\label{lem:ShermanMorrison}
    Let $A\in \R^{n\times n }$ be regular and $\bm{u}, \bm{v} \in \R^n$ be two column vectors. Then $A + \bm{uv}^\top$ is invertible, iff $1 + \bm{v}^\top A^{-1} \bm{u} \neq 0$. Then,
    \begin{equation*}
        (A + \bm{uv}^\top)^{-1} = A^{-1} - \frac{A^{-1}\bm{uv}^\top A^{-1}}{1+\bm{v}^\top A^{-1}\bm{u}}.
    \end{equation*}
\end{lemma}
The third lemma summarizes the determinant and inverse update formulas needed to evaluate the objective increment for the proof of the upcoming theorem.
\begin{lemma}[Determinant and Inverse under Rank-1 Updates]
\label{lem:rank1-updates}
Let $\bm{u}\in\R^{n}$ and $\Sigma_{i+1}=\Sigma_i+\bm{u}\bm{u}^\top$ and $a = \bm{u}^\top \Sigma_i^{-1} \bm{u}.$ Then
\begin{align*}
    \log|\Sigma_{i+1}| &= \log|\Sigma_i| + \log(1+a), \quad \text{and}\\
    \Sigma_{i+1}^{-1}
    &= \Sigma_i^{-1}
    - \frac{\Sigma_i^{-1}\bm{u} \bm{u}^\top \Sigma_i^{-1}}{1+a}.
\end{align*}
\end{lemma}
\begin{proof}
We begin with the determinant identity and get
\begin{align*}
    \log | \Sigma_{i+1}| &= \log | \Sigma_i + \bm{u} \bm{u}^\top| \\ 
    \overset{\text{Lemma \ref{lem:matrixDeterminant}}}&{=} \log \left(|\Sigma_i| ( 1+ \bm{u}^\top \Sigma_i^{-1} \bm{u}) \right)\\
    &= \log|\Sigma_i| + \log\left(1 + \bm{u}^\top \Sigma_i^{-1} \bm{u}\right).
\end{align*}
Setting $a = \bm{u}^\top \Sigma_i^{-1}\bm{u}$ yields the result.

The inverse update follows from the Sherman–Morrison formula applied to
$\Sigma_{i}$, i.e. 
\begin{align*}
    \Sigma_{i+1}^{-1} = \left(\Sigma_i + \bm{u}\bm{u}^\top\right)^{-1}
    \overset{\text{Lemma}~\ref{lem:ShermanMorrison}}{=} \Sigma_i^{-1} - \frac{\Sigma_i^{-1}\bm{u} \bm{u}^\top \Sigma_i^{-1}}{1+ \bm{u}^\top \Sigma_i^{-1}\bm{u}} 
\end{align*}
Using the definition of $a$ leads to the second equation and concludes the proof. 
\end{proof}

In the following theorem, we give the exact per-step change of $\mathcal{L}(\widetilde{K})$ induced by the
rank-1 update~\eqref{eq:candidate-update}.

\begin{theorem}[Additive Decomposition of the Exact Rank-1 VFE Change]
\label{thm:per-pivot-gain}
Let $\sigma_\varepsilon^2>0$, $\widetilde{K}_i \in \R^{n\times n}$ with $\widetilde{K}_i\preceq K$ and define $\Sigma_i = \widetilde{K}_i + \sigma_\varepsilon^2I$.
For any rank-1 update $\widetilde K_{i+1}=\widetilde K_i+\bm{u} \bm{u}^\top$ with $\bm{u}\in\R^n$, the exact per-step change in the VFE admits the closed-form additive decomposition 
\begin{equation}
    \Delta \mathcal L
    :=
    \mathcal L(\widetilde K_{i+1})-\mathcal L(\widetilde K_i)
    =
    -\frac{1}{2}\left(
    \log(1+a)
    -\frac{b^2}{1+a}
    -\frac{1}{\sigma_\varepsilon^2}\|\bm{u}\|_2^2
    \right).
    \label{eq:deltaVFE}
\end{equation}
with $a := \bm{u}^\top \Sigma_i^{-1} \bm{u},$ and $b := \bm{u}^\top \Sigma_i^{-1} \bm y.$
\end{theorem}
\begin{proof}
    At iteration $i$, we have
    \begin{equation}
    \mathcal L(\widetilde K_i)
        =
        -\frac12\left(
        \log|\Sigma_i|
        +
        \bm y^\top\Sigma_i^{-1}\bm y
        +
        \frac{1}{\sigma_\varepsilon^2}\tr(T_i)
        \right).
    \label{eq:Li}
    \end{equation}
    Analogously, after applying the rank-1 update,
    \begin{equation}
        \mathcal L(\widetilde K_{i+1})
        =
        -\frac12\left(
        \log|\Sigma_{i+1}|
        +
        \bm y^\top\Sigma_{i+1}^{-1}\bm y
        +
        \frac{1}{\sigma_\varepsilon^2}\tr(T_{i+1})
        \right).
    \label{eq:Liplus1}
    \end{equation}
    Subtracting~\eqref{eq:Li} from~\eqref{eq:Liplus1} gives
    \begin{align}
        \Delta \mathcal L
        &=
        \mathcal L(\widetilde K_{i+1})-\mathcal L(\widetilde K_i) \nonumber \\
        &=
        -\frac12\left(
        \log|\Sigma_{i+1}|-\log|\Sigma_i|
        +
        \bm y^\top(\Sigma_{i+1}^{-1}-\Sigma_i^{-1})\bm y
        +
        \frac{1}{\sigma_\varepsilon^2}
        (\tr(T_{i+1})-\tr(T_i))
        \right),
    \label{eq:DeltaLDecomp}
    \end{align}
    where we evaluate the three differences separately.
    
    \noindent
    \emph{(i) Complexity term.}
    By the first property of Lemma~\ref{lem:rank1-updates},
    \begin{align*}
        \Delta\mathcal{L}_{\text{complexity}} &:= \log|\Sigma_{i+1}|-\log|\Sigma_i| =\log |\Sigma_i| + \log \left(1 + a \right) - \log| \Sigma_i|
        \\
        &=
        \log(1+a).
    \end{align*}
    
    \noindent
    \emph{(ii) Data-fit term.}
    Using the second property of Lemma~\ref{lem:rank1-updates}, we get
    \begin{align*}
        \Sigma_{i+1}^{-1} - \Sigma_i^{-1}
        &= 
        \Sigma_i^{-1}
        -
        \frac{\Sigma_i^{-1}\bm{u} \bm{u}^\top \Sigma_i^{-1}}{1+a} - \Sigma_i^{-1}
        = -\frac{\Sigma_i^{-1}\bm{u} \bm{u}^\top \Sigma_i^{-1}}{1+a}.
    \end{align*}
    Multiplying from left and right by $\bm y$ yields
    \begin{align*}
        \Delta \mathcal{L}_{\text{data-fit}} &:=  \bm y^\top(\Sigma_{i+1}^{-1}-\Sigma_i^{-1})\bm y
        =
        -
        \frac{\bm{y}^\top \Sigma_i^{-1}\bm{u}
        \bm{u}^\top \Sigma_i^{-1}\bm{y}}{1+a}
         \nonumber \\
        &=
        - \frac{(\bm{u}^\top \Sigma_i^{-1}\bm y)^2}{1+a}.
    \end{align*}
    Defining $b := \bm{u}^\top \Sigma_i^{-1}\bm y$ gives
    \begin{equation*}
        \Delta\mathcal{L}_{\text{data-fit}} = \bm y^\top(\Sigma_{i+1}^{-1}-\Sigma_i^{-1})\bm y
        =
        -
        \frac{b^2}{1+a}.
    \end{equation*}

    \noindent
    \emph{(iii) Trace term.}
    From $T_{i+1}=T_i-\bm{u}\bm{u}^\top$ and the linearity of the trace operator we obtain
    \begin{equation*}
        \Delta \mathcal{L}_{\text{trace}} = \tr(T_{i+1})-\tr(T_i)
        := -\tr(\bm{u}\bm{u}^\top)
        = -\bm{u}^\top \bm{u}
        =-\|\bm{u}\|_2^2.
    \end{equation*}
    Substituting $\Delta\mathcal{L}_{\text{complexity}}$, $\Delta\mathcal{L}_{\text{data-fit}}$ and $\Delta\mathcal{L}_{\text{trace}}$ into~\eqref{eq:DeltaLDecomp}
    yields
    \begin{equation*}
        \Delta \mathcal L
        =
        -\frac{1}{2}\left(
        \log(1+a)
        -
        \frac{b^2}{1+a}
        -
        \frac{1}{\sigma_\varepsilon^2}\|\bm{u}\|_2^2
        \right),
    \end{equation*}
    which leads to the desired result.
\end{proof}
From the theorem it becomes evident that each term of the exact per-step change $\Delta \mathcal{L}$ corresponds to one of the three terms of $\mathcal{L}(\widetilde{K})$ in~\eqref{eq:vfe}.

\paragraph{Cholesky-Consistent Updates}
Theorem~\ref{thm:per-pivot-gain} holds for any rank-1 update vector $\bm{u} \in \mathbb{R}^n$. To obtain a practical algorithm, we restrict our attention to updates that preserve the incremental Cholesky factorization $\widetilde{K}_i = L_i L_i^\top$, namely those of the form $\bm{u}_p = \bm{t}_p/\sqrt{(T_i)_{p,p}}$, i.e.~the column that pivoted Cholesky would append to $L_i$ when selecting pivot $p$. For the remainder of this work the index $p$ identifies quantities connected to pivot $p$.
Under this restriction, the gain formula~\eqref{eq:deltaVFE} reduces to a closed-form criterion in terms of the residual column $\bm{t}_p$ and diagonal entry $d_p^{(i)}$ that already appear in classical pivoted Cholesky.
The following proposition formalizes the construction and gives the explicit forms of $a_p$ and $b_p$.

\begin{proposition}[Cholesky-Consistent Rank-1 Updates]
\label{prop:chol-consistent}
At iteration $i$, let $\widetilde{K}_i = L_i L_i^\top$, $T_i = K - \widetilde{K}_i \succeq 0$,
and let $p$ satisfy $(T_i)_{pp} > 0$. For
\begin{equation*}
    \label{eq:chol-direction}
    \bm{t}_p  = (T_i)_{:,p}, \quad
    d_p^{(i)} = (T_i)_{p,p}, 
    \quad
    \text{and}
    \quad
    \bm{u}_p = \bm{t}_p / \sqrt{{d}_p^{(i)}}.
\end{equation*}
it holds 
\begin{itemize}
    \item[(i)] $L_{i+1} := (L_i \;\; \bm{u}_p)$ satisfies $\widetilde{K}_{i+1} = L_{i+1}L_{i+1}^\top$,
    \item[(ii)] the residual update is given by $T_{i+1} = T_i - \bm{u}_p \bm{u}_p^\top \succeq 0$ with $(T_{i+1})_{pp} = 0$, and
    \item[(iii)] the per-pivot quantities from Theorem~\ref{thm:per-pivot-gain} take the form
    \begin{equation}
    \label{eq:ajbj}
        a_p := \frac{\bm{t}_p^\top \Sigma_i^{-1} \bm{t}_p}{{d}_p^{(i)}},
        \quad
        \text{and}
        \quad
        b_p := \frac{\bm{t}_p^\top \Sigma_i^{-1} \bm{y}}{\sqrt{{d}_p^{(i)}}}.
    \end{equation}
\end{itemize} 
\end{proposition}
\begin{proof}
    The factorization identity (i) $\widetilde{K}_{i+1}=L_{i+1}L_{i+1}^\top$ follows from $L_i L_i^\top + \bm{u}_p \bm{u}_p^\top = \widetilde{K}_i +  \bm{u}_p \bm{u}_p^\top$. For (ii), the diagonal entry satisfies
    \begin{equation*}
        (T_{i+1})_{p,p}
        = (T_i)_{p,p} - \left(\bm{u}_p\right)_p^2
        = {d}_p^{(i)} - \frac{\left({d}_p^{(i)}\right)^2}{{d}_p^{(i)}}
        = 0,
    \end{equation*}
    where we used $(\bm{t}_p)_p = (T_i)_{p,p} = {d}_p^{(i)}$ and $(\bm{u}_p)_p = (\bm{t}_p)_p / \sqrt{{d}_p^{(i)}} = \sqrt{{d}_p^{(i)}}$. Further, we observe the residual update $T_{i+1} = T_i -\bm{u}_p \bm{u}_p^\top \succeq 0$ follows from the definitions of $T_i$ and $T_{i+1}$. Positive semi-definiteness (psd) follows from the standard Cholesky rank-1 argument, i.e.~since $T_i \succeq0$ and $(T_i)_{p,p} = 0$ the residual, after eliminating row and column $p$, is again psd. Regarding (iii), the forms of $a_p$ and $b_p$ follow by substituting $\bm{u}_p = \bm{t}_p/\sqrt{{d}_p^{(i)}}$ into the definitions $a_p = \bm{u}_p^\top \Sigma_i^{-1} \bm{u}_p$ and $b_p = \bm{u}_p^\top \Sigma_i^{-1} \bm{y}$ of Theorem~\ref{thm:per-pivot-gain}. This concludes the proof.
\end{proof}
Taken together, Theorem~\ref{thm:per-pivot-gain} and Proposition~\ref{prop:chol-consistent} yield a complete specification of the pivot-selection criterion. At iteration $i$, it evaluates $\Delta\mathcal{L}_p$ using \eqref{eq:ajbj} and selects for each admissible candidate $p$ the pivot that maximizes this gain. We refer to the resulting pivot rule, together with the Cholesky-consistent rank-1 update as \textit{$\Delta$-VFE pivoted Cholesky}.

A further property of Cholesky-consistent updates, beyond preserving the factorization, is that they guarantee monotonic improvement of $\mathcal{L}(\widetilde{K})$ along the update sequence, regardless of the specific pivot-selection rule. This is shown in the following proposition. 

\begin{proposition}[Monotonicity]
    \label{lem:monotone}
    Let $\sigma_\varepsilon^2>0$. At iteration $i$, suppose that
    $T_i=K-\widetilde K_i\succeq 0$ and let $p$ be any admissible pivot
    with $(T_i)_{pp}>0$. Then
    \begin{equation}
        \Delta\mathcal{L}_p
        \geq
        \frac{b_p^2}{2(1+a_p)}
        \geq 0.
        \label{eq:DeltaLowerBound}
    \end{equation}
    Consequently, any sequence of admissible Cholesky-consistent rank-one updates satisfies
    \begin{equation}
        \mathcal{L}(\widetilde{K}_{i})
        \leq
        \mathcal{L}(\widetilde{K}_{i+1})
        \leq
        \mathcal{L}(K)
        \label{eq:VFEmonoSeq}
    \end{equation}
    for all iterations for which the update is defined.
\end{proposition}

\begin{proof}
    Let $i$ denote the current iteration. Since $\widetilde K_i\succeq 0$, we have $\Sigma_i=\widetilde K_i+\sigma_\varepsilon^2 I \succeq \sigma_\varepsilon^2 I.$
    Hence, $\Sigma_i^{-1}\preceq \sigma_\varepsilon^{-2}I.$
    Therefore,
    \begin{equation*}
        a_p
        =
        \bm u_p^\top\Sigma_i^{-1}\bm u_p
        \le
        \frac{\|\bm u_p\|_2^2}{\sigma_\varepsilon^2}.
    \end{equation*}
    Since $a_p\ge 0$, the inequality $\log(1+a_p)\le a_p$ gives
    \begin{equation*}
        \log(1+a_p)
        \le
        a_p
        \le
        \frac{\|\bm u_p\|_2^2}{\sigma_\varepsilon^2}.
    \end{equation*}
    Substituting this into the gain formula yields
    \begin{equation*}
        \Delta\mathcal{L}_p
        =
        \frac12
        \left(
        \frac{\|\bm u_p\|_2^2}{\sigma_\varepsilon^2}
        -
        \log(1+a_p)
        +
        \frac{b_p^2}{1+a_p}
        \right)
        \ge
        \frac{b_p^2}{2(1+a_p)}
        \ge
        0,
    \end{equation*}
    which proves \eqref{eq:DeltaLowerBound}.
    
    If $p_i$ is the pivot selected at iteration $i$, then
    \begin{equation*}
        \mathcal L(\widetilde K_{i+1})
        =
        \mathcal L(\widetilde K_i)
        +
        \Delta\mathcal L_{p_i}
        \ge
        \mathcal L(\widetilde K_i).
    \end{equation*}
    Moreover, the Cholesky-consistent update preserves
    $T_i=K-\widetilde K_i\succeq0$, hence $\widetilde K_i\preceq K$.
    By the VFE lower-bound property~\cite{Titsias.2009},
    \begin{equation*}
        \mathcal L(\widetilde K_i)\le \mathcal L(K).
    \end{equation*}
    This proves \eqref{eq:VFEmonoSeq} and concludes the proof.
\end{proof}
\begin{remark}
    The bound \eqref{eq:VFEmonoSeq} holds for any admissible pivot, i.e.~any index $p$ with $(T_i)_{p,p} > 0$. Consequently, greedy pivoted Cholesky and RPCholesky also produce monotonically non-decreasing sequences of $\mathcal{L}(\widetilde{K})$ under Cholesky-consistent rank-1 updates. By construction, the $\Delta$-VFE rule then selects, at each step, the admissible pivot that maximizes this guaranteed gain. \hfill$\triangle$
\end{remark}

%% file: Sections/05_Efficient_Realisation.tex
\section{Algorithmic Realization of $\Delta$-VFE Pivoted Cholesky}\label{sec:realization}

The structure of the pivot score \eqref{eq:deltaVFE} allows for an efficient realization based on Woodbury identities and incremental low-rank updates. We first discuss the resulting pivot selection strategy, then present the corresponding algorithm, and finally give the computational complexity relative to the classical pivoted Cholesky method.

\subsection{Candidate Sampling}

A straightforward approach to choose the pivot at iteration $i$ is
\begin{equation}
    p_\star \in \argmax_p \Delta\mathcal L_p.
\label{eq:pivot-rule}
\end{equation}

Evaluating the bound for each of the $n$ candidate pivots would imply additional computational complexity of $\mathcal{O}(n^2i)$ per iteration, since applying $\Sigma_i^{-1}$ costs $\mathcal{O}(ni)$. Thus, evaluating this for $n$ candidate pivots leads to $\mathcal{O}(n^2i)$ and overall to $\mathcal{O}(n^2r^2)$ complexity, which is clearly infeasible and not compatible with classical pivoted Cholesky's $\mathcal{O}(nr^2)$ complexity. We therefore come to a central aspect on how this can be efficiently evaluated. 

To this end, we introduce a \emph{batch of candidate points} $\mathcal{S}_i$, $|\mathcal{S}_i| =:s  \ll n$ at each iteration, for which we evaluate the per-step change. Further, we introduce the \emph{set of pivots} $\mathcal{P}_i$ and the \emph{set of candidates} $\mathcal{C}_i$, which includes all points that have not been used as pivots, i.e.~$\mathcal{C}_i \cap \mathcal{P}_i = \emptyset$. In order to build this set we rely on a cheap way to select potentially strong candidates $\bm{u}_p$, by focusing on the respective diagonal residual ${d}_p$. For each candidate $p \in \mathcal{C}_i$, we assign a selection probability proportional to its diagonal residual $d_p^{(i)}$,
\begin{equation}
    \Pr \{P = p\} = \frac{{d}_p^{(i)}}{\sum_{k\in\mathcal{C}_i} {d}_k^{(i)}}
\end{equation}
and draw $s$ indices from $\mathcal{C}_i$ according to this probability to form the batch $\mathcal{S} \subseteq \mathcal{C}$. Therefore, we use the same probability distribution as RPCholesky. Hence, for (and only for) $s=1$ our proposed method and RPCholesky naturally yield the same results.

\subsection{Woodbury-based Inverse Updates}
Efficient evaluation of the $\Delta$-VFE pivot score requires repeated application of the regularized inverse $\Sigma_i^{-1} = (\widetilde K_i + \sigma_\varepsilon^2 I)^{-1}.$ To avoid recomputing this inverse after every rank-1 update, we exploit the low-rank structure of $\widetilde K_i$ via the Sherman-Morrison-Woodbury formula. We first recall the corresponding matrix inversion identity.
\begin{lemma}[Matrix Inversion Lemma {\cite[Chapter 18, Section 2]{harville.1998}}]\label{lem:woodbury}
For regular matrices $A\in \R^{n\times n}$, $C \in \R^{k\times k}$, $U\in \R^{n \times k}$ and $V\in \R^{k\times n}$ the Woodbury matrix identity is
\begin{equation*}
    (A + UCV)^{-1} = A^{-1} - A^{-1}U (C^{-1} + VA^{-1}U)^{-1}VA^{-1}.
\end{equation*}   
\end{lemma}

Computing $\Sigma_i^{-1}$ from scratch at each iteration would require an $\mathcal{O}(n^3)$ factorization, rendering repeated evaluation of the quantities $a_p$ and $b_p$ in~\eqref{eq:deltaVFE} computationally prohibitive. However, since $\Sigma_i = L_iL_i^\top + \sigma_\varepsilon^2 I$ consists of a diagonal term plus a rank-$i$ correction, its inverse can instead be maintained implicitly using the Woodbury identity from Lemma~\ref{lem:woodbury}. Applying $\Sigma_i^{-1}$ to a vector then requires only $\mathcal{O}(ni)$ operations, with dominant cost arising from multiplications involving the low-rank factor $L_i$.
The following lemma shows that the corresponding inverse representation can furthermore be updated without recomputing matrix factorizations. We postpone the complexity analysis to Proposition~\ref{prop:complexity} in Subsection~\ref{subsec:complexity}.

\begin{lemma}
\label{lem:woodbury_update}
Let $\Sigma_i = L_i L_i^\top + \sigma_\varepsilon^2 I$ with $L_i \in \R^{n\times i}$. We define $M_i := (I + L_i^\top L_i/\sigma_\varepsilon^{2})\in \R^{i\times i}$ and its inverse $B_i := M_i^{-1}$.
After appending column $\bm{u}_p$ to form $L_{i+1} = (L_i \;\; \bm{u}_p)$, we define
\begin{equation}
    \bm{c} := \frac{L_i^\top \bm{u}_p}{\sigma^2_\varepsilon}\in\R^i, 
    \quad
    \xi := 1 + \|\bm{u}_p \|_2^2/\sigma_\varepsilon^2\in\R, \quad \text{and} \quad
    \gamma := (\xi - \bm{c}^\top B_i \bm{c})^{-1} \in \R.
\end{equation}
Then $B_{i+1} = (I + L_{i+1}^\top L_{i+1}/\sigma_\varepsilon^2)^{-1}$
is given by
\begin{equation}
    B_{i+1} =
    \begin{pmatrix}
        B_i + \gamma B_i \bm{c} \bm{c}^\top B_i & -\gamma B_i \bm{c} \\
        -\gamma \bm{c}^\top B_i            &  \gamma
    \end{pmatrix},
\end{equation}
where $\gamma > 0$.
\end{lemma}
The proof follows by a Schur complement argument and is given in Appendix~\ref{app:woodburyInverse}.

Taken together the candidate sampling as discussed in the previous subsection restricts the overall evaluation of our per-pivot gain to a set of candidates, whereas the  Woodbury inverse updates allow an efficient evaluation and update of the inverse. 

\subsection{The $\Delta$-VFE Pivoted Cholesky Algorithm}
We now state the resulting $\Delta$-VFE pivoted Cholesky algorithm. To this end, we separate the Woodbury application of $\Sigma_i^{-1}$ and the main algorithm and give these as Algorithms~\ref{alg:apply_sigma_inv} and \ref{alg:vfe_batch_pc} respectively.

\paragraph{Algorithm~\ref{alg:apply_sigma_inv} (Woodbury Apply)}
Algorithm~\ref{alg:apply_sigma_inv} applies the regularized inverse $\Sigma^{-1} = (LL^\top + \sigma_\varepsilon^2 I)^{-1}$ to a vector $\bm v \in \mathbb R^n$ without forming the inverse explicitly. Exploiting the low-rank structure of $\Sigma$, the Woodbury identity from Lemma~\ref{lem:woodbury} yields
$\Sigma^{-1}\bm v
=
{\bm v}/{\sigma_\varepsilon^{2}}
-
L B L^\top \bm v/\sigma_\varepsilon^4,$
where $B = \left(I + L^\top L/\sigma_\varepsilon^{2}\right)^{-1} \in \mathbb R^{r\times r}.$
The matrix $B$ is maintained incrementally using the update from Lemma~\ref{lem:woodbury_update}, so that each application of $\Sigma^{-1}$ requires only multiplications involving the low-rank factor $L$.
\begin{algorithm}[t]
\caption{\textsc{Woodbury Apply} $\Sigma^{-1}\bm{v}$}
\label{alg:apply_sigma_inv}
\begin{algorithmic}[1]
\REQUIRE Cholesky factor $L\in\mathbb{R}^{n\times i}$, matrix $B=(I+L^\top L/\sigma_\varepsilon^{2})^{-1}\in\mathbb{R}^{i\times i}$, $\bm{v}\in\mathbb{R}^n$, noise variance $\sigma_\varepsilon^2 >0$
\ENSURE $\bm{z}=\Sigma^{-1}\bm{v}$,
  where $\Sigma=LL^\top + \sigma_\varepsilon^2 I $
\RETURN $\bm{z} \leftarrow \bm{v}/\sigma_\varepsilon^{2}
  - LBL^\top\bm{v}/\sigma_\varepsilon^{4}$
\end{algorithmic}
\end{algorithm}

\paragraph{Algorithm~\ref{alg:vfe_batch_pc} ($\Delta$-VFE Pivoted Cholesky)}

\begin{algorithm}[t]
\caption{\textsc{$\Delta$-VFE Pivoted Cholesky}}
\label{alg:vfe_batch_pc}
\begin{algorithmic}[1]
\REQUIRE spd matrix $K\in\mathbb{R}^{n\times n}$,
output data $\bm{y}\in\mathbb{R}^n$, noise variance $\sigma_\varepsilon^2$,
batch size $s$, target rank $r$
\ENSURE Cholesky factor $L\in\mathbb{R}^{n\times r}$, pivots $\mathcal{P}=\{p_1,\dots,p_r\}$
\STATE $L \leftarrow [\,]$, \quad $\mathcal{P} \leftarrow \emptyset$
\STATE $B \leftarrow [\,]$
\COMMENT{Represents $(I+L^\top L/\sigma_\varepsilon^{2})^{-1}$ or a factorization}
\STATE $\bm{d} \leftarrow \operatorname{diag}(K)$
\COMMENT{Residual diagonal for sampling}
\FOR{$i = 1, \dots, r$}
    \STATE $\mathcal{C} \leftarrow \{1,\dots,n\} \setminus \mathcal{P}$
    \STATE $w_i \leftarrow {d}_i / \sum_{k \in \mathcal{C}} {d}_k$ for $i \in \mathcal{C}$
    \COMMENT{Sampling weights}
    \STATE Sample $\mathcal{S} \subset \mathcal{C}$ with
    $|\mathcal{S}| = \min(s, |\mathcal{C}|)$ using probabilities $\{w_i\}$
    \STATE $\bm{v} \leftarrow \textsc{WoodburyApply}(L, B, \bm{y})$
    \COMMENT{Algorithm~\ref{alg:apply_sigma_inv}}
    \FOR{$p \in \mathcal{S}$}
        \STATE $\bm{t}_p \leftarrow K_{:,p} - L L_{p,:}^\top$
        \STATE $\bm{u}_p \leftarrow \bm{t}_p / \sqrt{d_p}$
        \STATE $\bm{z}_p \leftarrow \textsc{WoodburyApply}(L, B, \bm{u}_p)$
        \STATE $a_p \leftarrow \bm{u}_p^\top \bm{z}_p$
        \STATE $b_p \leftarrow \bm{u}_p^\top \bm{v}$
        \STATE $\Delta\mathcal{L}_p \leftarrow
        -\tfrac{1}{2}\!\left(
        \log(1 + a_p) - {b_p^2}/(1 + a_p) - {\|\bm{u}_p\|^2}/{\sigma_\varepsilon^{2}}
        \right)$
    \ENDFOR
    \STATE $p_\star \leftarrow \argmax_{p \in \mathcal{S}}\, \Delta\mathcal{L}_p$
    \STATE $\bm{t} \leftarrow K_{:,p_\star} - L L_{p_\star,:}^\top$
    \STATE $\bm{u} \leftarrow \bm{t} / \sqrt{d_{p_\star}}$
    \STATE $\xi \leftarrow 1 + \|\bm{u}\|_2^2/\sigma_\varepsilon^2$
    \STATE $\gamma \leftarrow \left(\xi - \bm{c}^\top B\bm{c}\right)^{-1}$
    \STATE $B \leftarrow
    \begin{pmatrix}
    B + \gamma B\bm{c}\bm{c}^\top B & -\gamma B\bm{c} \\
    -\gamma \bm{c}^\top B & \gamma
    \end{pmatrix}$
    \COMMENT{Rank-1 block update}
    \STATE $L \leftarrow \bigl(L \;\; \bm{u}\bigr)$
    \STATE $\bm{d} \leftarrow \bm{d} - \bm{u}^2$
    \STATE $\mathcal{P} \leftarrow \mathcal{P} \cup \{p_\star\}$
\ENDFOR
\RETURN $L, \, \mathcal{P}$
\end{algorithmic}
\end{algorithm}

Algorithm~\ref{alg:vfe_batch_pc} implements the $\Delta$-VFE pivot rule within the Cholesky framework. The algorithm maintains three quantities, i.e. the Cholesky factor $L \in \mathbb{R}^{n \times r}$, the auxiliary matrix $B = (I + L^\top L/\sigma_\varepsilon^{2})^{-1} \in\mathbb{R}^{r \times r}$ needed for Woodbury applies, and the residual diagonal $\bm{d} = \mathrm{diag}(T_i)$ used for candidate sampling (lines~1 to~3).

At each iteration $i$, a batch $\mathcal{S}$ of $s$ candidate pivots is drawn without replacement from the candidate indices $\mathcal{C} = \{1,\ldots,n\} \setminus \mathcal{P}$, proportionally to the residual diagonal $d$ (lines~5 to~7). This is the same diagonal-proportional sampling used by RPCholesky. The vector $\bm{v} = \Sigma_i^{-1}\bm{y}$ is then computed once per iteration via Algorithm~\ref{alg:apply_sigma_inv} (line~8) and reused for all candidates in the batch.

For each candidate $p \in \mathcal{S}$, the residual column $\bm{t}_p = (K - LL^\top)\bm{e}_p$ is formed (line~10), the normalised direction $\bm{u}_p = \bm{t}_p/\sqrt{d_p}$ is computed (line~11) and a second Woodbury apply gives $\bm{z}_p = \Sigma_i^{-1} \bm{u}_p$ (line~12). The scalar quantities $a_p = \bm{u}_p^\top\bm{z}_p$ and $b_p = \bm{u}_p^\top\bm{v}$ (lines~13 and~14) are then combined to evaluate the per-pivot gain $\Delta\mathcal{L}_p$ (line~15) from Theorem~\ref{thm:per-pivot-gain}.

The pivot with the highest gain $p_\star = \argmax_{p \in \mathcal{S}}\,\Delta\mathcal{L}_p$ is selected (line~17). Its residual column and normalized direction are recomputed (lines~18 and~19). 
In order, to update $B$ via the rank-1 block formula of Lemma~\ref{lem:woodbury_update} (line~22), we update the parameters $\xi$ and $\gamma$ (lines~20 and~21).
The column $\bm{u}$ is then appended to $L$ (line~23, before the residual diagonal is updated as $\bm{d} \leftarrow \bm{d} - \bm{u}^2$ (line~24), where $\bm{u}^2$ denotes the element wise square. Finally, $p_\star$ is added to the pivot set (line~25).
The algorithm then returns the low-rank Cholesky factor $L$ as well as the list of pivots $\mathcal{P}$ (line~27).

\begin{remark}
Although Algorithm~\ref{alg:vfe_batch_pc} is written in terms of candidate-wise quantities $\{\bm t_p,\bm u_p,\Delta\mathcal L_p\}_{p\in\mathcal S}$ for clarity of exposition, the computationally dominant operations are implemented in batched form. In particular, the residual directions $\bm t_p=(K-LL^\top)\bm e_p, \text{~with~} p\in\mathcal S,$ are assembled jointly into a matrix representation, and the associated contractions with $L^\top$, $B$, and $\bm y$ are evaluated via matrix--matrix operations. This replaces repeated matrix--vector computations by batched linear algebra kernels, improving practical efficiency and hardware utilization while preserving the same asymptotic complexity. The complexity analysis below reflects this batched implementation.
\hfill $\triangle$
\end{remark}

\subsection{Complexity Analysis}\label{subsec:complexity}
Classical pivoted Cholesky and its randomized variant are of $\mathcal{O}(nr^2)$ complexity in time \cite{chen.2025, epperly.2025, harbrecht.2012}. We turn towards the analysis of our proposed pivot choice and state the per iteration and overall complexity in the following proposition.
\begin{proposition}[Complexity of $\Delta$-VFE pivoted Cholesky]
    \label{prop:complexity}
    Assume that at iteration $i$ the current Cholesky factor is $L_i \in \R^{n\times i}$, the batchsize is $s$ and columns of $K$ are formed on demand. Then,
    \begin{enumerate}
        \item[(i)] Applying $\Sigma_i^{-1}$ to a vector using Algorithm~\ref{alg:apply_sigma_inv} costs $\mathcal{O}(ni + i^2)$.
        \item[(ii)] One iteration of Algorithm~\ref{alg:vfe_batch_pc} costs $T_i = \mathcal{O}\left(s(ni+i^2+n) + ni + i^2 + n\right).$
        In particular, for $n \gg i$ and $n \gg s$ this reduces to $T_i = \mathcal{O}(sni)$.
        \item[(iii)] After $r$ iterations, the total runtime is $\mathcal{O}(snr^2 + sr^3 + snr)$
        and for $n \gg r$ this turns to $\mathcal{O}(snr^2)$.
        \item[(iv)] The space complexity is $\mathcal{O}(nr + r^2),$ excluding storage of a precomputed kernel matrix $K$.
    \end{enumerate}
\end{proposition}
We give the full derivation of the complexity analysis in Appendix~\ref{app:complexity}. In case that the kernel matrix $K$ needs to be precomputed, the time and space complexity rises to $\mathcal{O}(n^2d)$ and $\mathcal{O}(n^2)$, respectively. Given the above proposition, the total, asymptotic time and space complexity of our proposed method matches the one of classic pivoted Cholesky and RPCholesky up to the batching factor $s$. This is under the assumption that $n \gg r$, which is a common one for the efficient applicability of low-rank approximations.

%% file: Sections/06_Interpretation.tex
\section{Interpretation \& Relations to Existing Methods}\label{sec:interpretation}

In this section we interpret the structure of the $\Delta$-VFE pivot criterion and clarify its relation to classical pivoted Cholesky, ridge leverage scores, and other data-aware selection rules.

\subsection{Structure of the $\Delta$-VFE Gain}

For a candidate pivot $p$, the exact per-step change of $\mathcal{L}(\widetilde{K})$ is given by
\begin{equation*}
    \Delta\mathcal{L}_p = -\frac{1}{2}\left(\log(1+a_p)-\frac{b_p^2}{1+a_p} - \frac{1}{\sigma_\varepsilon^2} \|\bm{u}_p\|_2^2 \right),
\end{equation*}
where $a_p = \bm{u}_p^\top \Sigma_i^{-1} \bm{u}_p,$ and $b_p = \bm{u}_p^\top \Sigma_i^{-1} \bm{y}$, taking complexity, data-fit and trace penalty into account.

\paragraph{Inverse-Energy Interpretation of $a_p$}
Since $\Sigma_i^{-1} \succ 0$, it induces an inner product. Therefore, we can rewrite $a_p = \bm{u}_p^\top \Sigma_i^{-1} \bm{u}_p =  \|\bm{u}_p \|_{\Sigma_i^{-1}}^2$, i.e.~as the squared norm of the update in the $\Sigma_i^{-1}$-induced inner product. It quantifies how strongly the candidate direction interacts with the metric induced by the regularized inverse. The logarithmic term penalizes updates that reinforce directions already strongly weighted by the current inverse, regularized kernel matrix $\Sigma_i^{-1}$, thereby limiting redundant contributions to model complexity. 
Geometrically, interpreting $\Sigma_i^{-1}$ as a regularized inverse operator, we can also relate it to a Mahalanobis-type norm induced by $\Sigma_i^{-1}$~\cite{bishop.2006, christmann.2008}.

\paragraph{Residual Interpretation of $b_p$}
To interpret the data-fit term, we refer to Section~\ref{sec:basics}, where we have defined the approximate predictive mean of a GPR as
\begin{equation}
    \bm{\tilde\mu}_i = \widetilde K_i \Sigma_i^{-1} \bm{y}, \label{eq:mutilde}
\end{equation}
representing the prediction of $\bm{y}$ under the rank-$i$ approximation.
Since $\Sigma_i = \widetilde K_i + \sigma_\varepsilon^2 I,$ we have $(\widetilde K_i + \sigma_\varepsilon^2 I)\Sigma_i^{-1} = I$. 
Applying this identity to $\bm{y}$ yields
\begin{equation*}
    \widetilde K_i \Sigma_i^{-1} \bm{y}
    + \sigma_\varepsilon^2 \Sigma_i^{-1} \bm{y}
    =
    \bm{y}. \label{eq:RLSdeltaVFE}
\end{equation*}
Inserting \eqref{eq:mutilde}, this becomes $\bm{\tilde\mu}_i+\sigma_\varepsilon^2 \Sigma_i^{-1} \bm{y}=\bm{y},$ and therefore
\begin{equation*}
    \Sigma_i^{-1} \bm{y}
    =
    \frac{1}{\sigma_\varepsilon^{2}}
    (\bm{y} - \bm{\tilde\mu}_i).
\end{equation*}
Consequently, the data-fit term can be written as
\begin{equation}
    b_p
    =
    \bm{u}_p^\top \Sigma_i^{-1} \bm{y}
    =
    \frac{1}{\sigma_\varepsilon^{2}}
    \bm{u}_p^\top (\bm{y} - \bm{\tilde\mu}_i). \label{eq:bReformulated}
\end{equation}
Thus, the alignment term is a scaled inner-product between the candidate pivot direction and the current prediction residual, measuring their unnormalized alignment.

\paragraph{Trace Interpretation of $\bm{u}_p$}
The third term $\|\bm{u}_p\|_2^2/\sigma_\varepsilon^2$ admits a direct algebraic interpretation. By Lemma~\ref{lem:trace}, the one-step decrease in $\tr(T_i)$ induced by selecting pivot $p$
is
\begin{equation*}
    \tr(T_i) - \tr(T_{i+1})
    = \|\bm{u}_p\|_2^2
    = \frac{(T_i^2)_{pp}}{(T_i)_{pp}},
\end{equation*}
so $\|\bm{u}_p\|_2^2/\sigma_\varepsilon^2$ corresponds to the one-step decrease in $\tr(T_i)/\sigma_\varepsilon^2$, the trace penalty component of $\mathcal{L}(\widetilde{K})$ as we have shown in Lemma~\ref{lem:trace}. 
The $\Delta$-VFE criterion therefore incorporates the \emph{exact} one-step trace decrease rather than the diagonal proxy as classical pivoted Cholesky does.

\subsection{Connection to Existing Work}
\label{subsec:connections}

The $\Delta$-VFE gain formula $\Delta\mathcal{L}_p$  decomposes into three terms, each with a precise  algebraic interpretation that connects to a distinct strand of existing work. We discuss these connections in turn.

\paragraph{Classical Pivoted Cholesky and Three-Level Hierarchy}
The term $\|\bm{u}_p\|_2^2/\sigma_\varepsilon^2$ in $\Delta\mathcal{L}_p$ equals the exact one-step  decrease in $\tr(T_i)/\sigma_\varepsilon^2$ by Lemma~\ref{lem:trace}. This term places three pivot rules on a hierarchy of successive refinements. Classical greedy pivoted Cholesky and RPCholesky operate at the first level, using the diagonal entry $(T_i)_{p,p}$ as a tractable proxy for the trace decrease, as established in Section~\ref{sec:pcholesky}. The  exact trace-reduction maximizer of Lemma~\ref{lem:trace}  sits at the second level, using the full column $\bm{t}_p$ but still ignoring the log-determinant and data-fit contributions. At the third level, $\Delta$-VFE combines the exact trace term with $\log(1+a_p)$ and $b_p^2/(1+a_p)$. These contributions go beyond residual geometry, as the first depends on the regularized inverse $\Sigma_i^{-1}$ and the second additionally on the data vector $\bm{y}$. Each level incorporates additional information beyond the preceding one, ranging from diagonal residual information to full residual geometry and finally to the complete variational objective.

\paragraph{Relation to Sparse GP Optimization}
Cao et al.~\cite{cao.2015} obtain the same scalar per-candidate variational gain in the setting of sparse Gaussian process inducing-point optimization, through an augmented QR representation of the partial Cholesky factor. The Woodbury-based derivation given in Theorem~\ref{thm:per-pivot-gain} expresses the gain directly in terms of the regularized inverse $\Sigma_i$ and the Schur-complement geometry of the residual. This formulation makes the additive decomposition of the gain into log-determinant, data-fit, and trace contributions explicit, and positions $\Delta\mathcal{L}_p$ as a pivot-selection criterion on the same footing as the residual-based criteria above. The structural relation to classical greedy pivoted Cholesky and RPCholesky described in the previous paragraph is a consequence of this form and is not visible in the augmented QR representation of Cao et al.

\paragraph{Relation to Data-Aware Pivot Rules}
Both Schreiter et al.~\cite{schreiter.2016} and our data-fit term $b_p = \bm{u}_p^\top \Sigma_i^{-1} \bm{y}$ in $\Delta\mathcal{L}_p$ exploit the predictive residual $\bm{y} - \bm{\tilde{\mu}}_i$ and the current sparse model to guide pivot selection. However, they enter the selection rule with different roles in the selection rule. The criterion proposed by Schreiter et al.~is pointwise, i.e. each candidate $p$ is scored by the magnitude of the predictive residual at a single index $|y_p - \tilde{\mu}_{i,p}|$. In contrast, $b_p$ is a directional quantity. Following~\eqref{eq:bReformulated}, it is, up to the constant $1/\sigma_\varepsilon^2$, the inner product of the candidate update direction $\bm{u}_p$ with the current residual vector $\bm{y} - \bm{\tilde{\mu}}_i$. As a result, candidates are scored not by where the residual is largest at a single index, but by how strongly the corresponding rank-1 update aligns with the unexplained part of the data. Moreover, $b_p$ enters $\Delta\mathcal{L}_p$ alongside the complexity and trace terms, whereas Schreiter's criterion is used as a stand-alone selection rule. 

\paragraph{Relation to Ridge Leverage Scores}
Ridge leverage scores~\cite{alaoui.2015, musco.2017} sample columns proportionally to the diagonal of $K(K + \sigma_\varepsilon^2 I)^{-1}$, measuring global point importance under the full kernel matrix. The data-fit term $b_p$, despite its appearance in the reformulation~\eqref{eq:bReformulated}, is structurally different. It is a directional quantity constructed from the current low-rank approximation $\widetilde{K}_i$ rather than the full kernel $K$, depends on the data vector $\bm{y}$, and is updated incrementally as pivots are added. The two criteria emphasize different aspects of approximation quality and are not directly aligned.

%% file: Sections/07_experiments.tex
\section{Numerical Experiments}\label{sec:experiments}
\sloppy
In this section, we compare classical pivoted Cholesky, RPCholesky and $\Delta$-VFE pivoted Cholesky on several datasets.
We first describe the methodology and implementation details. 
The code to reproduce the results shown in this section can be found in our GitHub repository \url{https://github.com/SM4DA/DeltaVFE_Pivoted_Cholesky}.

The experiments in this section are designed to illustrate the structural effects predicted by the gain formula of Theorem~\ref{thm:per-pivot-gain} on standard benchmark problems. This is done on the following four targets: (i) relative VFE error, (ii) relative trace-norm error, (iii) predictive performance in GP regression and (iv) influence of the batch size $s$.

\subsection{Experimental Setup}
Each dataset is given as $\mathcal{D} = \{(\bm{x}_i,y_i)\}_{i=1}^n$, where $n$ denotes the total number of data points. For each dataset we evaluate all methods for different target ranks $r$ up to an maximal target rank $r_{\text{max}}=2\,048$ and average the results over 15 runs for every examined method involving randomization, i.e.~$\Delta$-VFE pivoted Cholesky and RPCholesky.

For each dataset, we use 80\% of the total $n$ data points for relative VFE and approximation errors.  The remaining 20\% will be used to evaluate predictive quantities of GPR, as introduced in Section~\ref{sec:basics}. 
The  hyperparameters, i.e.~the kernel parameter, are optimized for each dataset independently following \cite[Section 5.2]{CarlEdwardRasmussenandChristopherK.I.Williams.}, i.e. by maximizing the LML.

\subsection{Datasets}
We give a short introduction to the datasets used for the comparison of the methods. 
\paragraph{Abalones}
The Abalone dataset \cite{abalone} is a well-known dataset used for benchmarking in machine learning context, such as the work of Gittens et al.~\cite{gittens.2013}, where it was used for the analysis of different sampling strategies for Nystr\"om. This data was originally collected in order to predict the age of abalones from physical measurements and contains 4\,177 data points for eight features. We use the Gaussian kernel as introduced in Definition~\ref{def:kernels}. The corresponding standard deviation and the mean of the randomized methods is depicted in each of the respective figures. We set the batch size $s=8$ for $\Delta$-VFE pivoted Cholesky. 

\paragraph{QM7}
The QM7 dataset \cite{blum.qm7, rupp.qm7} is a benchmark regression dataset from quantum chemistry consisting of 7\,165 data points for 273 features. Each data point corresponds to a small molecule and is represented by a Coulomb matrix descriptor \cite{rupp.qm7}. The target variable is a scalar-valued energy. 
For this application we use the Laplacian kernel as specified in Definition~\ref{def:kernels}. The dataset provides a predefined partition into five folds for 5-fold cross-validation. For the analysis of the relative VFE error and the relative trace-norm error, we set $s=16$ for $\Delta$-VFE pivoted Cholesky.
We further use this dataset to analyze the influence of the choice of the batch size $s$ in Subsection~\ref{subsec:batch size}.

\subsection{VFE \& Approximation Error Comparison}\label{sec:perf}

We compare greedy pivoted Cholesky, RPCholesky, and $\Delta$-VFE pivoted Cholesky with respect to relative VFE error and relative trace-norm error as functions of the target rank $r$ in Figure~\ref{fig:vfe-approx}. The panels in the top row are associated to the Abalone dataset, and the bottom row to QM7. 

\begin{figure}
    \centering
    \includegraphics[width=\linewidth]{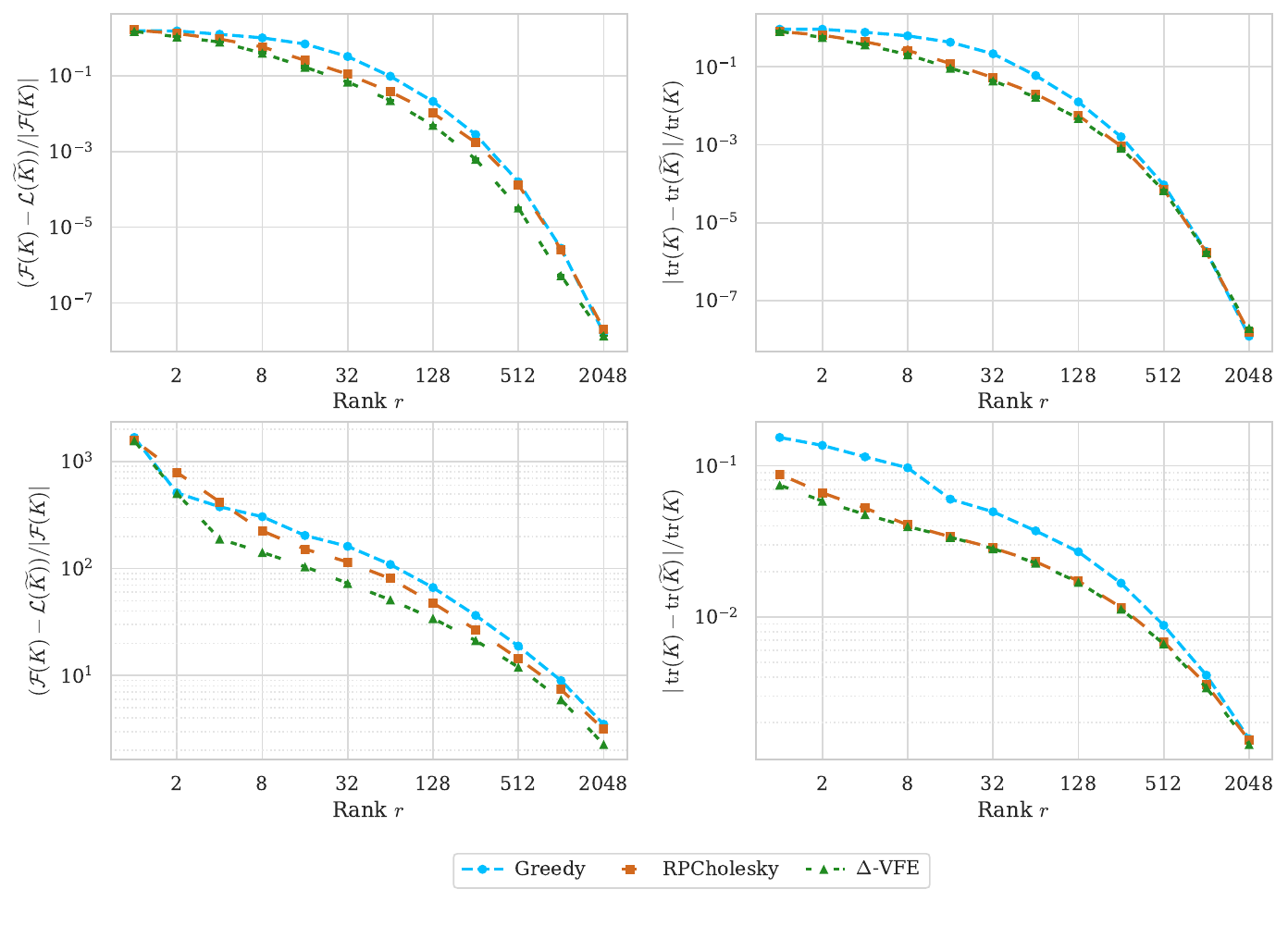}
    \caption{Comparison of greedy pivoted Cholesky, RPCholesky and $\Delta$-VFE pivoted Cholesky on the Abalone (top row) and QM7 (bottom row). We set the the batch parameter of $\Delta$-VFE to $s=8$ and $s=16$ for Abalone and QM7 respectively. The comparison of the respective relative VFE errors are shown in the left column. We display a comparison of the relative trace-norm error in the right column.}
    \label{fig:vfe-approx}
\end{figure}

\paragraph{VFE Objective}
On both datasets, $\Delta$-VFE attains the smallest relative VFE gap $(\mathcal{F}(K) - \mathcal{L}(\widetilde{K}))/|\mathcal{F}(K)|$ among the compared methods across all examined ranks, where a smaller value corresponds to a tighter approximation of the LML $\mathcal{F}(K)$.

For the Abalone dataset (Figure~\ref{fig:vfe-approx}, top left panel), a clear hierarchy of the compared methods emerges, where $\Delta$-VFE consistently outperforms both RPCholesky and greedy pivoted Cholesky. This advantage is especially noticeable for target ranks up to $r=32$. For the target rank $r=8$, corresponding to approximately $0.19\%$ of the number of datapoints $n$, $\Delta$-VFE attains a relative gap of $42.63\%$, reducing the gap of greedy pivoted Cholesky by approximately $57.79\%$ and that of RPCholesky by approximately $26.91\%$. 
The hierarchy persists across all ranks up to convergence, with $\Delta$-VFE's gap reduction relative to greedy pivoted Cholesky remaining above $70\%$ throughout $r \in [16, 1024]$. Only for $r=2048$ do all three methods drop below a relative gap of $0.02\%$, at which point the curves also visually coincide.

On QM7 (Figure~\ref{fig:vfe-approx}, bottom left panel), the same overall hierarchy can be observed, although the separation between the methods is less pronounced at the smallest ranks. In particular, for $r \in \{1,2\}$ all three methods attain very large relative gaps and remain comparatively close on the scale of the plot, with the hierarchy only becoming clearly distinguishable from $r=4$ onward. At $r=4$, $\Delta$-VFE reduces the gap of greedy pivoted Cholesky by approximately $50.43\%$ and that of RPCholesky by approximately $49.01\%$. The advantage over greedy pivoted Cholesky peaks at approximately $57.04\%$ for $r=64$. Even at the largest evaluated rank, $r=2048$, $\Delta$-VFE still achieves the smallest relative gap, namely $230.66\%$, compared to $360.89\%$ for greedy pivoted Cholesky and $317.43\%$ for RPCholesky, indicating that none of the methods has fully converged within the evaluated rank range. As the rank increases further, the gap between $\Delta$-VFE and RPCholesky gradually decreases, reflecting that randomized pivot selection becomes increasingly competitive once sufficiently many pivots are available.

Overall, these results validate that the specifically designed $\Delta$-VFE pivot rule yields superior performance on the VFE objective and demonstrate that both the complexity and data-fit terms play an important role in pivot selection.

\paragraph{Kernel Matrix Approximation Error}
The relative trace-norm error $|\tr(K)-\tr(\widetilde{K})|/\tr(K)$ is shown in the right column of Figure~\ref{fig:vfe-approx}. On the Abalone dataset (Figure~\ref{fig:vfe-approx}, top right panel), all methods start from a similar relative trace-norm error and display the same hierarchy as we have observed it for $\mathcal{L}(\widetilde{K})$, from the smallest ranks up to convergence. $\Delta$-VFE pivoted Cholesky achieves results better or identical to RPCholesky throughout. The largest gap from $\Delta$-VFE to all the other methods is given at $r=32$, where $\Delta$-VFE attains a 15.88\% and 79.45\% lower relative trace-norm error than RPCholesky and classical pivoted Cholesky.   

On QM7 (Figure~\ref{fig:vfe-approx}, bottom right panel), the same hierarchy is visible. Throughout $\Delta$-VFE attains the lowest relative trace-norm error at the smallest ranks, with the gap to RPCholesky reaching 15.19\% at $r=1$ and 9.17\% at $r=4$, before the two methods become visually indistinguishable across the mid rank regime. The gap to greedy pivoted Cholesky remains substantial throughout, exceeding 30\% for $r\leq 256$ and peaking at $r=8$ with 59.06\%.

Since $\Delta$-VFE uses the same diagonal-residual sampling distribution as RPCholesky, and reduces exactly to RPCholesky for $s=1$, it can be interpreted as an objective-aware generalization of RPCholesky. While this does not imply a priori dominance in trace error, the experiments show that $\Delta$-VFE maintains trace approximation quality comparable to RPCholesky while improving the target objective at low to moderate rank. This behaviour is consistent with the gain formulation in Theorem~\ref{thm:per-pivot-gain}, where the data-fit term $b_p$ contributes most when the current approximation $\widetilde{K}_i$ is far from $K$ and the residual $\bm{y} - \tilde{\bm{\mu}}_i$ contains significant structure. As the rank increases and $\widetilde{K}_i \to K$, the advantage diminishes and all methods approach the same objective value, which is consistent with Proposition~\ref{lem:monotone}.

\subsection{Prediction Error}\label{subsec:predError}

\begin{figure}[t]
    \centering
    \includegraphics[width=\linewidth]{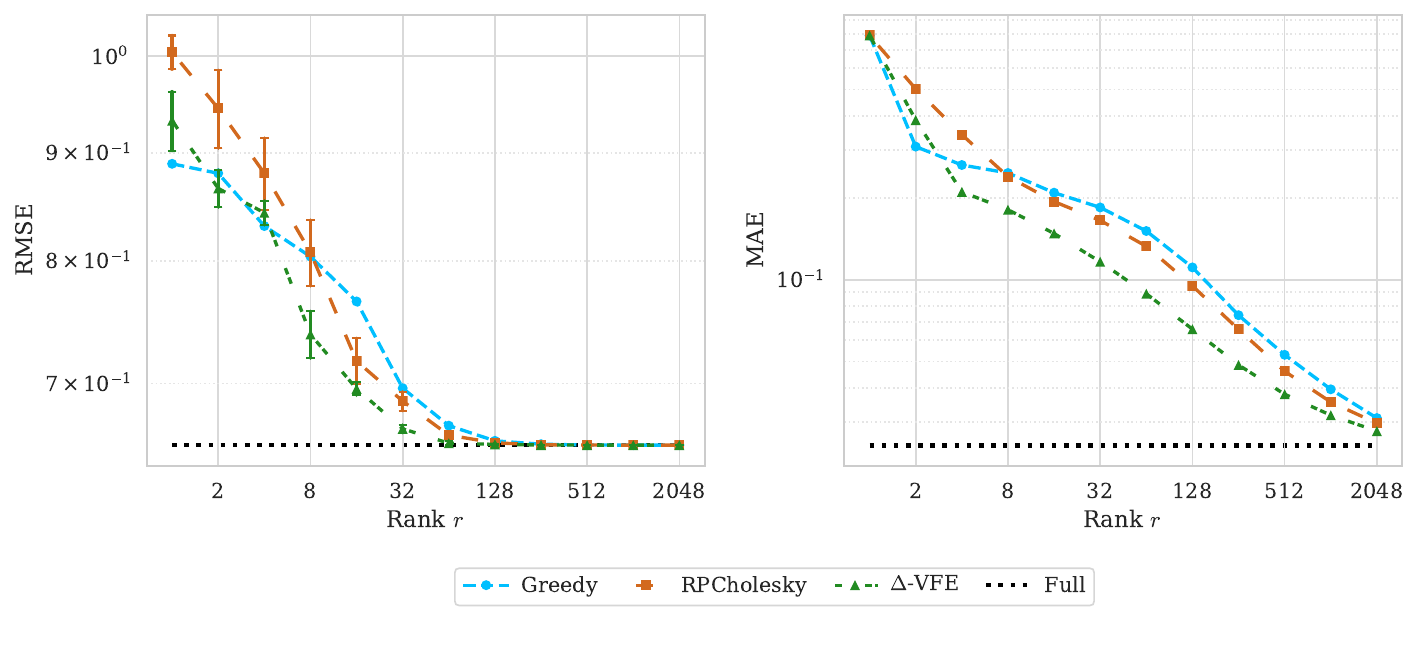}
    \caption{Comparison of the respective prediction errors on Abalone (left) using the Gaussian kernel and QM7 (right) using the Laplacian kernel. We include the baseline, which is obtained by using all $n$ data points for the respective datasets.}
    \label{fig:qm7abalone_prediction}
\end{figure}

We evaluate the predictive performance on the held-out 20\% test data points using the \textit{root mean square error} (RMSE) for Abalone and \textit{mean absolute error} (MAE) for QM7, the latter being standard for this dataset~\cite{rupp.qm7}. These are defined as
\begin{equation*}
    \text{RMSE} = \sqrt{\frac{1}{m}\sum_{i=1}^m (y_{i,\star} - \mu_{\star,i})^2} \quad \text{and} \quad \text{MAE}= \frac{1}{m} \sum_{i=1}^m |y_{\star,i} - \mu_{\star,i}|,
\end{equation*}
where  $y_{\star,i}$ denotes the $i$-th component of the true output $\bm{y}_\star$ and $\bm{\mu}_\star$ the prediction obtained from the respective low-rank approximation as defined in Section~\ref{sec:basics}, and $m$ is the number of test data points. The results are shown in Figure~\ref{fig:qm7abalone_prediction}

On Abalone (Figure~\ref{fig:qm7abalone_prediction}, left panel), greedy pivoted Cholesky attains the lowest RMSE at $r=1$ and is matched by $\Delta$-VFE pivoted Cholesky at $r=2$ and $r=4$. From $r=8$ onwards $\Delta$-VFE yields a significantly lower RMSE than both greedy pivoted Cholesky and RPCholesky, with the largest gap occurring at $r=8$, where the RMSE is 5.55\% and 5.98\%, respectively. From $r=16$ the hierarchy observed in the previous analysis reappears and persists up to convergence. 

On QM7 (Figure~\ref{fig:qm7abalone_prediction}, right panel), all methods yield a similar MAE at $r=1$, and greedy pivoted Cholesky yields the lowest MAE at $r=2$. The rate of convergence then reduces, and at $r=4$, $\Delta$-VFE achieves a 20.43\% lower MAE than greedy pivoted Cholesky and a 38.37\% lower MAE than RPCholesky. From $r=8$ onwards, the hierarchy of the previous analysis is recovered. 

The analysis on these two dataset reveals that our data-aware pivot rule yields improved errors for the GPR task, with the largest improvements at low to moderate ranks. 

\subsection{Influence of the Batch Size}\label{subsec:batch size}

\begin{figure}
    \centering
    \includegraphics[width=\linewidth]{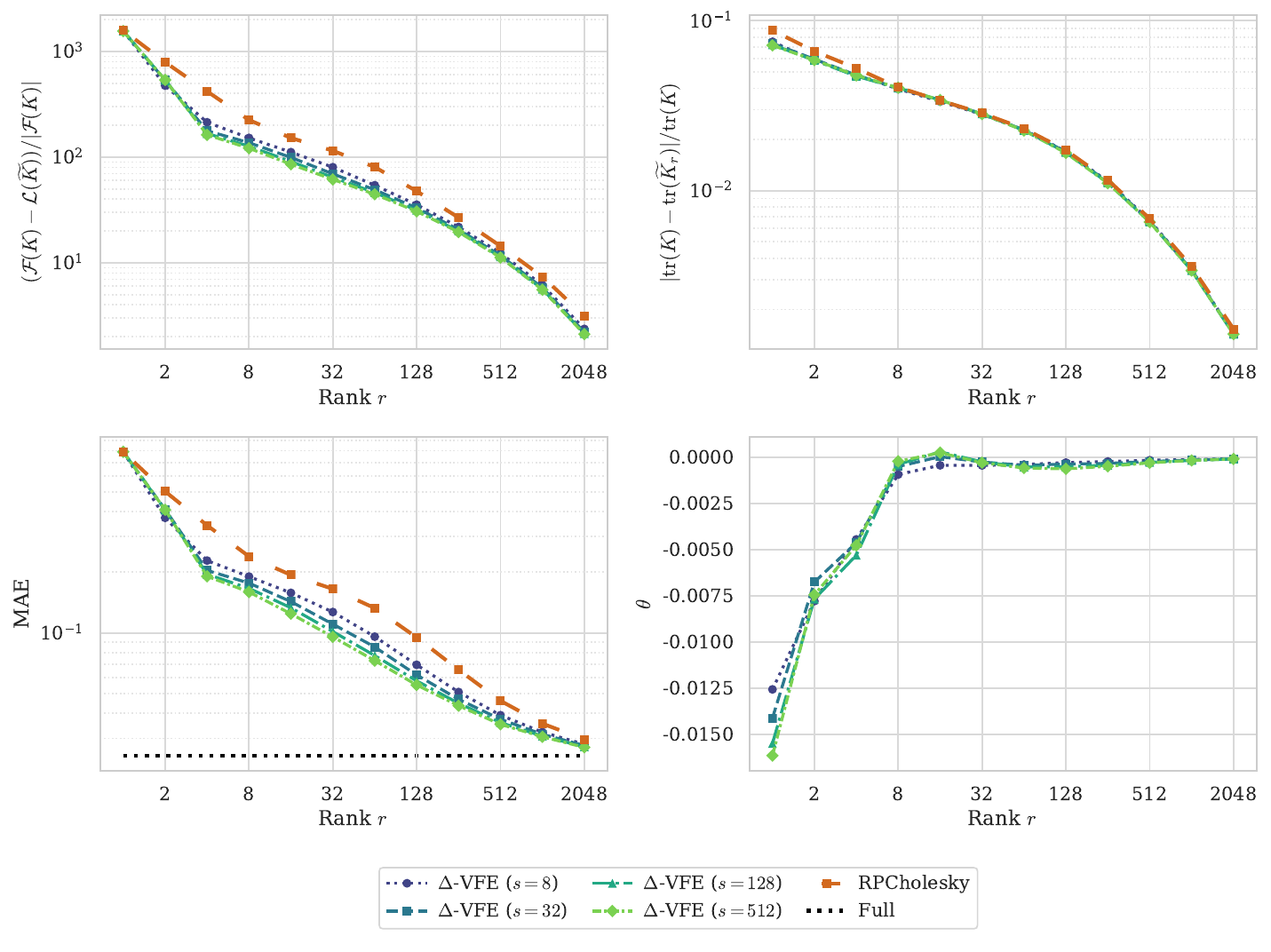}
    \caption{The influence of choosing different batch sizes $s$ for the $\Delta$-VFE pivot rule using $s=8,~32,~128,~512$ on QM7. The top left panel shows the evaluation of $\mathcal{L}(\widetilde{K})$ compared to the full baseline $\mathcal{F}(K)$. The top right panel shows the relative trace-norm error, while the bottom right panel offers a deeper insight by displaying the difference of the achieved trace-norm errors of $\Delta$-VFE and RPCholesky as $\theta$. The bottom left panel compares the predictive performance, where we have included the baseline that we get by using all $n$ data points.}
    \label{fig:batchsize}
\end{figure}

Finally, we examine the influence of the batch size $s$ on the QM7 dataset, displayed in Figure~\ref{fig:batchsize}. We use the batch sizes $s \in \{8,\,32,\,128,\,512\}$. As noted in Section~\ref{sec:realization}, $\Delta$-VFE reduces to RPCholesky for $s=1$. Therefore, we display not only $\Delta$-VFE pivoted Cholesky, but RPCholesky as a reference as well. The whole comparison is given in Figure~\ref{fig:batchsize}.

We first consider the results of the relative VFE gap (Figure~\ref{fig:batchsize}, top-left panel). All batch sizes yield a smaller relative gap than RPCholesky across all evaluated ranks, with the difference relative to RPCholesky dominating the differences between batch sizes. Among the $\Delta$-VFE configurations, the relative gap decreases monotonically with $s$ at every rank. Among the $\Delta$-VFE configurations, the relative gap is generally smaller for larger $s$, with the ordering becoming consistent from $r=4$ onwards. The largest spread between batch sizes occurs at $r=4$, where the relative gap for $s=8$ is 31.47\% larger than for $s=512$ and that for $s=32$ is 9.15\% larger.

The relative trace-norm error (Figure~\ref{fig:batchsize}, top right panel), behaves similarly, with the curves for different batch sizes barely distinguishable. The largest deviation between batch sizes occurs at $r=1$, where $s=8$ shows an approximation error roughly 5\% higher than $s=512$. We give a detailed view of this behaviour in the bottom right panel, by plotting difference of the relative trace-norm errors  between the $\Delta$-VFE realizations and RPCholesky, i.e.
\begin{equation*}
    \theta := \frac{|\tr(K)-\tr(\widetilde{K}_{\Delta-\text{VFE}})|-|\tr(K)-\tr(\widetilde{K}_{\text{RPCholesky}})|}{\tr(K)}.
\end{equation*}
A negative value of $\theta$ indicates that $\Delta$-VFE pivoted Cholesky attains a lower approximation error than RPCholesky. The batch size $s=8$ yields strictly lower trace-norm errors than RPCholesky across the displayed range, while $s=32$, $s=128$ and $s=512$ are slightly worse than RPCholesky at target rank $r=16$. However, from $r=32$ onwards, all batch sizes yield trace-norm errors at most equal to RPCholesky.

Regarding the prediction error (Figure~\ref{fig:batchsize}, bottom left panel), all batch sizes attain lower MAE values than RPCholesky. From $r=4$ onwards, larger batch sizes yield lower MAE montonically in $s$. Nevertheless, with diminishing returns. The gain from $s=8$ to $s=32$ is larger than the gain we obtain from $s=128$ to $s=512$.

Taken together, these results demonstrate that increasing $s$ yields improvements in the relative VFE gap and predictive performance with diminishing returns, while even the smallest batch size $s=8$ already attains a substantial fraction of the achievable gain over RPCholesky and also outperfomrs it over all metrics. These findings support choosing $s\ll n$ in practice. In this case $s$ acts as a small constant within the overall time complexity $\mathcal{O}(snr^2)$, which we derived in Section~\ref{sec:realization}, rather than a factor scaling with $r$ or even with $n$.

%% file: Sections/08_conclusion_outlook.tex
\section{Conclusion}

Classical greedy and randomized pivoted Cholesky factorizations select pivots to reduce the trace norm of the residual $T_r = K - L_r L_r^\top$, a criterion that misaligns with applications where $K$ appears inside a nonlinear functional involving additional terms beyond the trace. We addressed this gap for the variational free energy $\mathcal{L}(\widetilde{K})$ by deriving the exact per-pivot gain $\Delta\mathcal{L}_p$ under Cholesky-consistent rank-1 updates in closed form. The gain decomposes additively into complexity, data-fit, and trace contributions, where the trace contribution coincides with the exact one-step trace reduction rather than the diagonal proxy used by classical pivoted Cholesky. Selecting pivots greedily by $\Delta\mathcal{L}_p$ produces a monotonically non-decreasing sequence of $\mathcal{L}(\widetilde{K})$ values and admits evaluation on a batch of s candidates at $\mathcal{O}(snr^2)$ cost via incremental Woodbury updates, matching the asymptotic complexity of RPCholesky up to the batch factor. Numerical experiments on Abalone and QM7 confirm that $\Delta$-VFE pivoted Cholesky consistently achieves higher $\mathcal{L}(\widetilde{K})$ values and improved predictive accuracy at low to moderate ranks across both kernel types, while preserving the kernel approximation quality of RPCholesky and improving over classical greedy pivoted Cholesky.
Several directions remain open. The theoretical analysis of convergence behaviour and the role of the batch size as an algorithmic parameter are natural next steps. The interaction between objective-aware pivot selection and hyperparameter choice, as well as the extension of the framework to other matrix functionals, are further directions of interest.
\newpage

%% file: Sections/acknowledgment.tex
\section*{Acknowledgments}

In order to improve the readability and language quality of this manuscript, DeepL Write was utilized during the writing process. It is important to note that the
content and ideas presented in the manuscript were not generated by this or similar tools, and the authors take full responsibility for the content. Their use was focused solely on improving fluency and correcting spelling errors. A final review and correction of the manuscript was performed by the authors to ensure accuracy and consistency.

%% file: Sections/09_appendix.tex
\section{Incrementally Updating Woodbury Inverse} \label{app:woodburyInverse}
To keep the complexity of algorithm~\ref{alg:vfe_batch_pc} in a competitive regime to existing methods, we update $B$ incrementally, which is specified in Lemma~\ref{lem:woodbury_update}. This is done via the Schur complement. We give the exact derivation and complexity in the following.
\begin{proof}[Proof of Lemma~\ref{lem:woodbury_update}]
    For $L_{i+1}^\top L_{i+1}$ holds
    \begin{equation*}
        L_{i+1}^\top L_{i+1} = 
        \begin{pmatrix}
            L_i^\top L_i & L_i^\top \bm{u}_p \\
            \bm{u}_p^\top L_i & \bm{u}_p^\top \bm{u}_p
        \end{pmatrix}
    \end{equation*}
    and we get 
    \begin{equation*}
        M_{i+1} = \left(I +  L_{i+1}^\top L_{i+1}/\sigma_\varepsilon^{2} \right) =
        \begin{pmatrix}
            I + L_i^\top L_i/\sigma_\varepsilon^{2} & L_i^\top \bm{u}_p/\sigma_\varepsilon^{2} \\
            \bm{u}_p^\top L_i/\sigma_\varepsilon^{2}  & 1 + \|\bm{u}_p\|_2^2/\sigma_\varepsilon^2
        \end{pmatrix}
        =
        \begin{pmatrix}
            M_i & \bm{c} \\
            \bm{c}^\top & \xi
        \end{pmatrix},
    \end{equation*}
    with $\bm{c} := L_i^\top\bm{u}_p/\sigma_\varepsilon^2$ and $\xi := 1 + \|\bm{u}_p\|_2^2/\sigma_\varepsilon^2$.
    Since $M_i \succ 0$, we use the Schur complement block inverse and obtain
    \begin{equation*}
        M_{i+1}^{-1} = 
        \begin{pmatrix}
            M_i^{-1} + M_i^{-1}\bm{c}\psi^{-1}\bm{c}^\top M_i^{-1} & - M_i^{-1}\bm{c}\psi^{-1} \\
            -\psi^{-1}\bm{c}^\top M_i^{-1} & \psi^{-1} 
        \end{pmatrix},
    \end{equation*}
    where $\psi := \xi - \bm{c}^\top M_i^{-1} \bm{c}$ is the Schur complement of $M_i$ in $M_{i+1}$. Substituting $B_i := M_i^{-1}$ and $\gamma := \psi^{-1}$ yields
    \begin{equation*}
        B_{i+1} = 
        \begin{pmatrix}
            B_i + \gamma B_i\bm{c}\bm{c}^\top B_i & - \gamma B_i\bm{c} \\
            -\gamma\,\bm{c}^\top B_i & \gamma 
        \end{pmatrix},
    \end{equation*}
    which is the desired form. Finally, since $L_{i+1}^\top L_{i+1} \succeq 0$, we have $M_{i+1} = I + L_{i+1}^\top L_{i+1}/\sigma_\varepsilon^2 \succ 0$. The Schur complement of a positive block in a positive-definite matrix is itself positive, so $\psi > 0$ and hence $\gamma = \psi^{-1} > 0$.
\end{proof}

\section{Derivation of Complexity of $\Delta$-VFE pivoted Cholesky}\label{app:complexity}
In the following, we give the derivation of complexity 
\paragraph{Time Complexity}
We start with Algorithm~\ref{alg:apply_sigma_inv}, which involves a matrix-vector multiplication of $L^\top \in \R^{i \times n}$ and $\bm{v}\in\R^n$ in line 1 and is of complexity $\mathcal{O}(ni)$. In line 2 a $(i\times i)$-dimensional matrix is multiplied with a $i$-dimensional vector and thus is of complexity $\mathcal{O}(i^2)$. Finally, we need one last matrix-vector multiplication $L\bm{v}$  for line 3, which gives once again $\mathcal{O}(ni)$. The overall complexity for algorithm~\ref{alg:apply_sigma_inv} therefore is $\mathcal{O}(ni + i^2)$. 

For the analysis of Algorithm~\ref{alg:vfe_batch_pc}, where we first start with initialization of the quantities to be maintained, including $\bm{d}$ in line 3. This requires a full-diagonal scan of the kernel matrix $K$, leading to $\mathcal{O}(n)$. We then look at the outer loop given by iteration $i$. Lines 5 to 7 are dominated by accesses and are therefore of $\mathcal{O}(n)$. In line 8, we apply algorithm~\ref{alg:apply_sigma_inv}, which is of cost $\mathcal{O}(ni+i^2)$. 

Turning towards the inner loop, i.e.~the calculations we need for each candidate pivot in the current batch $\mathcal{S}$. For each candidate $p \in \mathcal{S}$, we need the residual vector $\bm{t}_p$ in line 10, which includes forming the $p$-th column of $K$, which is of complexity $\mathcal{O}(n)$. Further, computing $LL_{p,:}^\top$ leads to $\mathcal{O}(ni)$. So the overall cost for $\bm{t}_p$ is $\mathcal{O}(ni+n)$. For $\bm{u}_p$ in line 11 we then need to scale an $n$/dimensional vector $\bm{t}_p$, which is of cost $\mathcal{O}(n)$. To obtain $\bm{w}_p$ we once again apply algorithm~\ref{alg:apply_sigma_inv}, which adds $\mathcal{O}(ni +i^2)$ to the cost per candidate pivot. For $a_p$ and $b_p$, we need two dot products that leads to $\mathcal{O}(n)$. Summing up everything for each of the $s$ candidates we get 
\begin{align*}
    \mathcal{O}((ni+n) + (ni + i^2) + n) = \mathcal{O}(ni + i^2 +n).
\end{align*}
Consequently, the overall cost for all candidates is
\begin{align*}
    \mathcal{O}(s(ni+i^2 + n)).
\end{align*}

After choosing $p_\star$, which is of complexity $\mathcal{O}(s)$, we need to recompute $\bm{t}$ and $\bm{u}$, leading to $\mathcal{O}(ni+n)$ for lines 18 and 19. We perform the update of $B$ incrementally in lines 20 to 22, we gave details on how this can be done in the appendix~\ref{app:woodburyInverse}. This adds another $\mathcal{O}(i^2)$ and updating the diagonal $\bm{d}$ in line 24 costs $\mathcal{O}(n)$. Hence, the update cost is
\begin{align*}
    \mathcal{O}(ni+i^2 +n).
\end{align*}
Combining everything gives us the per-iteration cost 
\begin{align*}
    T_i &= \mathcal{O}(n + (ni + i^2) + s(ni+i^2+n) + (ni + i^2 +n)) \\
    &= \mathcal{O}(sni+si^2 +sn + ni+ i^2 +n).
\end{align*}
We assume that $n\gg i$ and moderate $s$ and thus get the per-iteration dominant term $T_i = \mathcal{O}(sni)$.
To get the overall costs, we sum $T_i$ over $i=0,\hdots, r$.
\begin{align*}
    \sum_{i=0}^{r-1}\mathcal{O}(sni) = \mathcal{O}\left( sn\sum_{i=0}^{r-1}i\right) = \mathcal{O}\left(sn\frac{r(r-1)}{2} \right) = \mathcal{O}(snr^2).
\end{align*}
Then the next order terms are 
\begin{align*}
    \sum_{i=0}^{r-1} \mathcal{O}(si^2) = \mathcal{O}(sr^3), \quad \sum_{i=0}^{r-1}\mathcal{O}(sn) = \mathcal{O}(snr).
\end{align*}
Thus, the overall costs are
\begin{equation*}
    T_{\text{total}} = \mathcal{O}(snr^2 + sr^3 + snr) = \mathcal{O}(snr^2 + sr^3).
\end{equation*}
If $n \gg r$, which corresponds exactly to the case when we expect low-rank approximations to be successfully applied, then the leading term is
\begin{equation*}
    T_{\text{total}} = \mathcal{O}(snr^2).
\end{equation*}
Thus, we preserve the asymptotic computational structure of RPCholesky regarding time complexity up to a small batching factor $s$. We give insights on the influence of the batch size in Section~\ref{sec:experiments}.
\paragraph{Space Complexity}
During the execution of $\Delta$-VFE algorithm we must maintain the approximate Cholesky-factor $L \in \R^{n\times r}$ and $\bm{d} \in \R^n$, which are the same quantities that we need to maintain for pivoted Cholesky and RPCholesky. Additionally, we need to keep track of $B \in \R^{r\times r}$. Therefore, the overall space complexity of $\Delta$-VFE is $\mathcal{O}(nr + r^2)$. This is under the assumption that we evaluate the necessary entries of $K \in \R^{n \times n}$ on the fly, otherwise, if $K$ is precomputed the space complexity for all methods is $\mathcal{O}(n^2)$.